\documentclass[11pt,reqno,letter]{amsart}

\usepackage{amsmath,amsthm,verbatim,amscd,amssymb,setspace,enumitem,hyperref,exscale,color,slashed}
\usepackage[all,pdf]{xy}

\newcommand{\CC}{\mathbb{C}}
\newcommand{\RR}{\mathbb{R}}
\newcommand{\OO}{\mathcal{O}}
\newcommand{\T}{\mathcal{T}}
\newcommand{\im}{\operatorname{Im}}

\newcommand{\tf}{\mathfrak{t}}
\newcommand{\sff}{\mathfrak{s}}
\newcommand{\gf}{\mathfrak{g}}
\newcommand{\rf}{\mathfrak{r}}
\newcommand{\bff}{\mathfrak{b}}

\newcommand{\pf}{\mathfrak{p}}
\newcommand{\nf}{\mathfrak{n}}
\newcommand{\lf}{\mathfrak{l}}
\newcommand{\suf}{\mathfrak{su}}
\newcommand{\PP}{\mathbb{P}}
\newcommand{\ZZ}{\mathbb{Z}}

\newcommand{\tr}{\operatorname{tr}}
\newcommand{\diag}{\operatorname{diag}}
\newcommand{\Pic}{\operatorname{Pic}}

\newtheorem{thm}{Theorem}[section]
\newtheorem{pro}[thm]{Proposition}

\newtheorem{lem}[thm]{Lemma}

\newcounter{mtheorem}
\newtheorem{mtheorem}[mtheorem]{Theorem}

\theoremstyle{definition}
\newtheorem{define}[thm]{Definition}
\newtheorem{rem}[thm]{Remark}
\newtheorem{ex}[thm]{Example}

\numberwithin{equation}{section}

\title{Invariant scalar-flat K\"ahler metrics on\\ line bundles over generalized flag varieties}
\address{CUNY Graduate Center}
\email{qyao1@gradcenter.cuny.edu}
\author{Qi Yao}
\date{\today}

\begin{document}

\maketitle

\begin{abstract}
Let $G$ be a simply-connected semisimple compact Lie group, $X$ a compact K\"ahler manifold homogeneous under $G$, and $L$ a negative $G$-equivariant holomorphic line bundle over $X$. We prove that all $G$-invariant K\"ahler metrics on the total space of $L$ arise from the Calabi ansatz. Using this, we then show that there exists a unique $G$-invariant scalar-flat K\"ahler metric in each K\"ahler class of $L$.
\end{abstract}

\thispagestyle{empty}

\markboth{Qi Yao}{Invariant scalar-flat K\"ahler metrics on line bundles over generalized flag varieties}

\section{Introduction}

During the past a few decades, many works arise on explicit construction of scalar-flat K\"ahler metrics on noncompact 4-manifolds, usually with certain symmetry conditions. Relatively earlier works by Eguchi-Hanson \cite{eguchi1979self} and Lebrun \cite{lebrun1988} constructed a family of ALE scalar-flat K\"ahler metrics in the total space of $\OO(-n)$ with $U(2)$ symmetry. Then, in \cite{lebrun1991explicit}, Lebrun adapted the Gibbons-Hawking ansatz in hyperbolic model to construct Scalar-flat K\"ahler metrics with $S^1$ symmetry. Joyce \cite{joyce1995} extended Lebrun's hyperbolic ansatz to toric manifolds and Calderbank-Singer \cite{calderbank2004einstein} applied Joyce's construction to toric resolutions of $\CC^2/\Gamma$ with cyclic quotient and constructed a family of $T^2$-invariant ALE scalar-flat K\"ahler metrics. And Lock-Viaclovski \cite{lock2019smorgaasbord} generalized the results to deal with the minimal resolutions of $\CC^2/\Gamma$, where $\Gamma$ is a finite subgroup of $U(2)$ with no reflection. The existence of ALE scalar-flat K\"ahler metrics in small deformation of resolutions of $\CC^2/\Gamma$ also has been investigated by Honda \cite{honda2013deformation} \cite{honda2014scalar}, Lock-Viaclovski \cite{lock2019smorgaasbord} and Han-Viaclovsky \cite{han2019deformation}. However, in all of these examples, the uniqueness Theorem have been missing for a long time (unlike in the compact case \cite{berman2017convexity, chen2000space}, or in the noncompact case with cusps \cite{auvray2017space}). This is not even completely clear in the $U(2)$-invariant case considered by LeBrun. This paper is concerned with a generalization of LeBrun's existence results to a class of spaces with strong symmetry in all dimensions, and with a classification of scalar-flat K\"ahler metrics on these spaces under a symmetry assumption. 

A compact homogeneous K\"ahler manifold is a compact K\"ahler manifold $(X,\omega)$ on which the automorphism group acts transitively. The classification of this type of spaces has been learned for a long time. In \cite{besse2007einstein}, every compact, simply-connected homogeneous K\"ahler manifold is isomorphic, in the sense of homogeneous complex manifolds, to an orbit of the adjoint representation of some compact semisimple Lie group endowed with canonical complex structure. Then, the classification of compact homogeneous K\"ahler manifolds reduces to classify the orbit space of adjoint representation. In general, each compact homogeneous K\"ahler manifold is the product of a flat complex torus and a compact simply-connected homogeneous K\"ahler manifold. In this paper, we're only interested in the compact homogeneous K\"ahler manifolds without torus part, which we call generalized flag variety.

\begin{mtheorem} \label{kf1} Let $X$ be a generalized flag variety, $L$ be a negative homogeneous line bundle over $X$ with $p: L\rightarrow X$, the natural projection of line bundle. Then, all invariant K\"ahler metrics can be written as 
\begin{align*}
\omega=p^* \omega_X + dd^c \varphi(r),
\end{align*}
where $\omega_X$ is an invariant K\"ahler form on $X$ and $\varphi(r)\in C^\infty(L)$. 
\end{mtheorem}

Based on Theorem \ref{kf1}, we have identified the $G$-invariant K\"ahler metrics in a certain K\"ahler class over $L$ with a class of single variable function. To determine the complete K\"ahler metrics with constant scalar curvature in a certain K\"ahler class, the method of momentum construction developed in \cite{hwang2002momentum}, is applied, which is in fact to solve a second order ODE. The second order ODE turns out to be solved in three cases as $C=0$, $C<0$ and $C>0$. In conclusion, we have the following classification Theorem,

\begin{mtheorem} \label{main} Let $X$ be a generalized flag variety, $L$ be a negative homogeneous line bundle over $X$, then, for any real number $C$, in each K\"ahler class there exists a unique $G$-invariant K\"ahler metric with constant scalar curvature $C$,
\begin{enumerate}
\item[(i)] If $C=0$, there exists a unique complete scalar-flat K\"ahler metric defined over the whole bundle $L$. This metric is an asymptotically conical K\"ahler metric.
\item[(ii)] If $C<0$, there exists a unique complete K\"ahler metric with negative constant scalar curvature and it is defined over a uniform disk of each fibre. And along each fibre the restricted K\"ahler metric is asymptotic to a hyperbolic metric.
\item[(iii)] If $C>0$, there exists a unique K\"ahler metric with positive constant scalar curvature defined over the whole line bundle. This metric can be completed by adding a divisor at infinity. In general, the infinity points are cone singularities along each fiber. The completing manifold is smooth only for some special choice of $(L,\omega_X, C)$. 
\end{enumerate} 
\end{mtheorem}

This paper is arranged as follows. Section \ref{secflag} is dedicated to preliminaries. In Section \ref{classCHKM}, we classify the generalized flag varieties of a semisimple Lie group $G^\CC$. The Section \ref{classlb} classify the holomorphic line bundle over $X$ by related each line bundle with an integral weight. Then, in Section \ref{2f0},  we discuss the $G$-invariant K\"ahler forms on $X$ and give an explicit expression of curvature form of holomorphic line bundle over $X$. The Section \ref{sec:ddbar} is dedicated to the proof of Theorem \ref{kf1}.  In Sections \ref{lefinv}--\ref{ss:diffof1form}, we discuss the invariant one form over level set of $L$ and their exterior derivatives. Section \ref{ss:invddbar} uses this information to prove an invariant $dd^C$-lemma on $L$. We also construct counterexamples to the invariant $dd^C$-lemma if $L$ is not negative. Section \ref{ss:mainproof} applies these results to prove Theorem \ref{kf1}. Section \ref{sec:mom} is dedicated to the proof of Theorem \ref{main}. In the Section \ref{calabisec}, we give a quick introduction of momentum construction and the Section \ref{calabisec1} proves Theorem \ref{main}.

The author would like to thank Professor Hans-Joachim Hein and Professor Bianca Santoro for suggesting the problem, and for constant support, many helpful comments, as well as much enlightening conversation. This work is completed while the author is supported by graduate assistant fellowship in Graduate Center, CUNY.

\section{Geometry of compact homogeneous K\"ahler manifolds}\label{secflag}

In this section, we recall the geometry of compact homogeneous K\"ahler manifolds. In Section \ref{classCHKM} we shall discuss the classification of compact homogeneous K\"ahler manifolds in terms of generalized flag varieties. In Section \ref{classlb}, given a compact homogeneous K\"ahler manifold, we classify all holomorphic line bundles over it, which can be characterized by the integral weights of a certain root system. Section \ref{2f0} dedicates to classify all $G$-invariant K\"ahler forms on $X$ and all $G$-invariant K\"ahler forms can be written explicitly.

\subsection{Classification of compact homogeneous K\"ahler manifolds} \label{classCHKM}

We say that a compact K\"ahler manifold $X$ is \emph{homogeneous} if the identity component of its biholomorphic isometry group acts transitively on $X$. According to \cite{myers1939group}, this group is a compact Lie group on $X$. In the following, we always let $G$ denote the universal covering of this group, and we assume that $G$ is semisimple. Let $p$ be a distinguished point in $X$. Let $R$ be the isotropy group of $p$ and let $S$ be the identity component of the center of $R$. Then, by \cite{matsushima1957}:
\begin{align*}
\text{The centralizer of $S$ in $G$ is $R$.} \tag{$\ast$}
\end{align*}
So the maximal tori containing $S$ must be contained in $R$. We fix a maximal torus $T$ with $S\subset T\subset R$ and denote the corresponding Lie algebras by $\sff\subset\tf\subset\rf$. Let $\gf$ be the 
Lie algebra of $G$.

Recall some basic notations and facts on Lie algebras as well as the following notations, which will be used throughout this paper. Let $\gf^\CC,\tf^\CC$ be the complexification of $\gf,\tf$, respectively. Let $\Delta(\tf, \gf)$ denote the root system of $\gf^\CC$ with respect to the Cartan subalgebra $\tf^\CC$. The positive root system and simple root system are denoted by $\Delta^+(\tf, \gf)$ and $\Pi=\{\alpha_1,\ldots, \alpha_n\}$, respectively. Notice that $\gf^\CC$ admits a root space decomposition,
\begin{align*}
\gf^\CC= \tf^\CC \oplus  \sum_{\alpha\in \Delta^+(\tf,\gf)} (\gf_\alpha \oplus \gf_{-\alpha}).
\end{align*}
Recall that the Killing form of $\gf^\CC$ is an invariant nondegenerate bilinear form. The Killing form of $\gf^\CC$ induces an inner product on the real vector space generated by the root system \cite[Corollary 2.38]{knapp2013lie}, denoted as $(\cdot, \cdot)$ in the following. Besides, the Killing form is also negative definite when restricted to the real subspace $\gf$, hence the negative Killing form defines an inner product in $\gf$.  By \cite[Chapter VI.1]{knapp2013lie}, $\gf$ is the compact real form of $\gf^\CC$. We introduce a \textit{normalized adapted basis} $\{X_\alpha, Y_\alpha \}$ of each pair of root spaces $\gf_\alpha \oplus \gf_{-\alpha}$, which satisfies the following:
\begin{enumerate}
\item[(a)] Let $\gf_\alpha$ and $\gf_{-\alpha}$ denote the complex eigen-spaces corresponding to the roots $\alpha$ and $-\alpha$. Then
\begin{align}\label{rootvector1} 
X_\alpha-iY_\alpha= E_{\alpha}\in \gf_\alpha\quad\text{and} \quad X_\alpha+iY_\alpha=E_{-\alpha}\in \gf_{-\alpha}.
\end{align}	
\item[(b)] $X_\alpha$ and $Y_\alpha$ are normalized in the sense that,
\begin{align}\label{rootvector2}
	[X_\alpha,Y_\alpha]=- H_\alpha,
\end{align}
where $H_\alpha$ satisfies $\beta( H_\alpha)= (\beta,\alpha)/2i$ for each root $\beta$. Consider the inner product induced by the negative Killing form. One can easily check that $|X_\alpha|^2=|Y_\alpha|^2=1/2$. This normalization will be applied in Section \ref{lefinv} to calculate the differential of invariant 1-forms.
\end{enumerate}
Then the compact real form $\gf$ can be written as follows:
\begin{align}
\gf = \tf\oplus \sum_{\alpha \in \Delta^+(\tf,\gf)} \RR \langle X_\alpha, Y_\alpha \rangle.
\end{align}

We can classify compact homogeneous K\"ahler manifolds in terms of generalized flag varieties. Let $G^{\CC}$ be the complexification of $G$ ($G^{\CC}$ is the simply-connected Lie group with Lie algebra $\gf^\CC$, where we refer to \cite[Chapter III, Section 6.8]{bourbakilie} for more details). Consider a parabolic subgroup $P$ of $G^{\CC}$. Then $P$ is determined by a subset of the simple root system. More precisely, for $\Pi'\subset \Pi$, the corresponding Lie algebra $\pf$ can be decomposed as
\begin{align} \label{paragp}
\pf=\tf^\CC \oplus \sum_{\alpha\in \Delta^+(\tf,\gf)}\gf_{\alpha}\oplus \sum_{\alpha\in  \Pi' } \gf_{-\alpha}=\bff\oplus \sum_{\alpha\in  \Pi' } \gf_{-\alpha},
\end{align}
where  $\bff$ is the Lie algebra of the Borel subgroup, i.e.,
\begin{align*}
\bff=\tf^\CC \oplus \sum_{\alpha\in \Delta^+(\tf,\gf)} \gf_{\alpha}.  
\end{align*} 
The \emph{generalized flag variety} with data $(G, \Pi, \Pi')$ is defined to be the complex manifold $X = G^\CC / P$.

Consider the maximal   compact subgroup $G$ of $G^\CC$. Then $G$ acts transitively on $X$ with stabilizer group $R=G \cap P$. The Lie algebra of the stabilizer group $R$ is
\begin{align} \label{lieofr}
\rf  = \tf \oplus \sum_{\alpha \in \Pi'} \RR \langle X_\alpha, Y_\alpha \rangle. 
\end{align}
Then the generalized flag variety $X$ can also be written as $G/R$, and its complex structure can also be described as follows. Let 
\begin{align}\label{tanroot}
D^+:=\Delta^+(\mathfrak{t},\mathfrak{g})\setminus  \Pi'.
\end{align}
Then $D^+$ is a \textit{closed} subset of the root system in the sense that for any $\alpha, \ \beta \in D^+ $, if $\alpha +\beta$ is a root, then $\alpha +\beta \in D^+$. The tangent space of $X$ at a distinguished point $p$ can be identified with 
\begin{align} \label{tanspace1}
T_p X = \sum_{\alpha \in D^+ } \RR \langle X_\alpha , Y_\alpha \rangle.
\end{align}
We also call $\{X_\alpha, Y_\alpha: \alpha \in D^+ \}$ a \textit{normalized adapted basis of $X$}. There exists a natural $R$-invariant almost complex structure $J$ on $T_p X$ given by
$$J(X_\alpha)=Y_\alpha, \quad J(Y_\alpha)=-X_\alpha.$$ Because $J$ is $R$-invariant, it extends to a $G$-invariant almost complex structure on the whole tangent bundle $TX$. The complexified tangent space  at $p$ splits into
\begin{align} \label{cxtandecom}
\displaystyle T^{(1,0)}_p X = \sum_{\alpha \in D^+} \gf_\alpha, \quad T^{(0,1)}_p X = \sum_{\alpha \in D^+} \gf_{-\alpha}.
\end{align}
One can check that $J$ is an integrable almost complex structure because $D^+$ is closed (see \cite[Section 12]{10.2307/2372795}), and the complex manifold $(X,J)$ is $G$-equivariantly biholomorphic to $G^\CC/P$.

Generalized flag varieties are simply-connected and the proof of this fact will be given in Lemma \ref{simconnected}. There exist K\"ahler forms on $(X,J)$, as discussed in section \ref{2f0}. Hence, each generalized flag variety is a simply-connected compact homogeneous K\"ahler manifold of $G$.  

Conversely, given a compact homogeneous K\"ahler manifold $X$ that admits a transitive holomorphic action by a simply-connected compact semisimple Lie group $G$ with data $(R, T, S)$ and a point $p \in X$ as above, let $\Delta(\tf, \rf)$ denote the root system of $\rf^\CC$ with respect to $\tf^\CC$. This can be viewed as a subset of $\Delta(\tf, \gf)$. Define
$$D := \Delta(\tf, \gf)\setminus\Delta(\tf, \rf).$$
The invariant complex structure $J$ on $X$ determines the set of positive roots $D^+$ as follows. Since the complexified tangent space of $X$ is identified with 
$$T_p^\CC X= \sum_{\gamma\in D} \gf_\gamma,$$
the $R$-invariance of $J$ implies that $J$ preserves each root space. Then $D^+$ is defined to be
$$D^+:= \{\gamma\in D : Jv=iv\;\,\text{for all}\;\,v\in \mathfrak{g}_\gamma\}.$$
The closedness of $D^+$ follows from the integrability of $J$. In addition, we can choose a simple root system $\Pi'$ in $\Delta(\tf,\rf)$. Then $D^+$ and the positive roots, $\Delta^+(\tf, \rf)$, generated by $\Pi'$ determine a positive root system $\Delta^+(\tf, \rf)$ in $\Delta(\tf, \gf)$, i.e.,
$$\Delta^+(\tf,\gf)= D^+ \cup \Delta^+ (\tf, \rf).$$ The set $\Pi'$ can be extended to a simple root system, $\Pi$, of $\Delta(\tf, \gf)$ such that $\Pi$ generates the positive roots $\Delta^+(\tf,\gf)$. 
More details on root systems can be found in \cite[Sections 13.6--13.7]{10.2307/2372795}. Based on this discussion, the Lie algebra $\rf$ can be written in terms of the simple root set $\Pi'$ as in (\ref{lieofr}). Thus, $X$ can be identified with the generalized flag variety associated with the data $(G, \Pi, \Pi')$.

The notion of a generalized flag variety is actually independent of the choice of a simple root system. By \cite[Section 5.13]{adams1982lectures}, the different simple root systems are identified by the action of the Weyl group, and then the associated generalized flag varieties are isomorphic via conjugation by an element of $G$. Hence, without loss of generality, we can fix a simple root system $\Pi$ at the beginning, and then each generalized flag variety of $G$ is classified in terms of a subset of $\Pi$.

In conclusion, we have proved the following theorem, which should be well-known to experts. 

\begin{thm}\label{classhs} Let $X$ be a compact K\"ahler manifold. Assume that $X$ is homogeneous under a simply-connected compact semisimple Lie group $G$. Fix a system of simple roots $\Pi$ of $\mathfrak{g}^\CC$. Then $X$ is $G$-equivariantly biholomorphic to the generalized flag variety of type $\Pi'$ for some subset $\Pi'\subset \Pi$, i.e.,
$$X \cong G^\CC/P, \;\text{where $P$ is the parabolic subgroup of $G^\CC$ determined by $\Pi'$}.$$
\end{thm}

\begin{rem}\label{simconnected}
In fact, all homogeneous manifolds discussed in this paper are simply-connected. More precisely, let $X$ be a compact K\"ahler manifold homogeneous under a simply-connected semisimple compact Lie group $G$. Then $X$ is simply-connected. This fact follows quickly from the fiber bundle $R \hookrightarrow G \to X$. Let $X=G/R$ and let $S$ be the connected center of $R$. According to the statement $(*)$, $R$ is the union of all the maximal tori containing $S$, which implies that $R$ is connected.  The fiber bundle structure
$R \hookrightarrow G \to X$
induces a long exact sequence of homotopy groups
\begin{align*}
    \cdots \rightarrow \pi_1(G) \rightarrow \pi_1(X) \rightarrow \pi_0(R)\rightarrow \cdots,
\end{align*}
and this tells us that $\pi_1(X)=0$.

According to \cite[Satz I]{borel1962kompakte}, each compact homogeneous K\"ahler manifold is the product of a flat complex torus and a simply-connected compact homogeneous K\"ahler manifold. Furthermore, the connected component of the identity of the automorphism group of $X$ is a semisimple Lie group \cite[Satz 4]{borel1962kompakte}. In conclusion, Theorem \ref{classhs} classifies all compact homogeneous K\"ahler manifolds without torus part.
\end{rem}

\subsection{Classification of holomorphic line bundles} \label{classlb}

Let $G$ be a simply-connected compact semisimple Lie group with complexification group $G^\CC$. Let $X$ be a compact K\"ahler manifold homogeneous under $G$. According to Theorem \ref{classhs}, as $X$ can be identified with a generalized flag variety of $G^\CC$, there is a natural $G^\CC$-action on $X$. Consider a holomorphic line bundle $L$ over $X$, with projection $\pi: L \rightarrow X$. The holomorphic line bundle is said to be \textit{$G^\CC$-homogeneous} (or in many articles \textit{$G^\CC$-linearizable}) if there exists a $G^\CC$-action on $L$ such that the projection $\pi$ is $G^\CC$-equivariant and the action is linear on fibers. In particular, we can construct $G^\CC$-homogeneous line bundle as follows.

Given a $1$-dimensional holomorphic representation $\chi:P\rightarrow \CC^*$, a $G^\CC$-equivariant holomorphic line bundle $L_\chi$ over $X = G^\CC/P$  can be constructed as follows:
\begin{align} \label{construcL}
L_\chi= G^\CC\times_\chi \CC= (G^\CC \times \CC)/{\sim},
\end{align}  
where $(gh, v)\sim(g,\chi(h)v)$ for all $g\in G^\CC$, $h\in P$, $v\in \CC$. This admits a natural $G^\CC$-action given by
\begin{align}\label{Gaction2}
g \cdot [(l,v)]=[(gl,v)]
\end{align}
for all $g\in G^\CC$, $l\in G^\CC$, $v \in \CC$. Conversely, given a $G^\CC$-equivariant holomorphic line bundle $L$ over $X$, the stabilizer group $P$ at the distinguished point $p\in X$ acts on the fiber $L_p$, inducing a holomorphic character $\chi: P \rightarrow \CC^*$ such that $L \cong L_\chi$. In conclusion, there is a one-to-one correspondence between the $G^\CC$-homogeneous line bundles over $X$ and the characters of $P$.

Let $L$ be an arbitrary holomorphic line bundle over $X$. We claim that $L$ is a $G^\CC$-homogeneous line bundle. To prove this claim, we need to borrow some results from algebraic geometry. Referring to \cite[Section 21.3]{humphreys2012linear}, $X$ is a projective variety. According to a well-known result from GAGA \cite{serre1956geometrie}, holomorphic line bundles over a projective variety are algebraic; in other words, the line bundle $L$ is algebraic over $X$. Moreover, since $G^\CC$ is a simply-connected semisimple Lie group, by \cite[Proposition 1]{popov1974picard}, the Picard group of $G^\CC$ is trivial. According to the key fact that $\Pic G^\CC =0 $, we can construct a character $\chi: P \rightarrow \CC^*$ such that $L \cong L_\chi$ (see \cite[Theorem 4 or Section 5]{popov1974picard} for details). Thus, we have the following proposition:

\begin{pro} \label{classlinebundle2} All holomorphic line bundles over $X$ are $G^\CC$-homogeneous.
\end{pro}

Fixing a simple root system $\Pi$ of $G$, let $X\cong G^\CC/ P$ be such that the parabolic group $P$ is determined by a subset of roots $\Pi'\subset \Pi$ as in (\ref{paragp}).  Let $\Pi=\{\alpha_1,\ldots, \alpha_n\}$ be ordered in such a way that $\alpha_i\in \Pi'$ if and only if $i=k+1,\ldots, n$. Let $\{\omega_1,\ldots, \omega_n\} \subset (\mathfrak{t}^\CC)^*$ be the set of fundamental weights corresponding to $\Pi$, which are defined by
\begin{align*}
	 \frac{2(\omega_i,\alpha_j)}{(\alpha_j,\alpha_j)}=\delta_{ij}\;\,(1\leq i,j\leq n).  
\end{align*}
Here the inner product is the bilinear form induced by the Killing form of $\gf$. The lattice generated by $\{\omega_1,\ldots,\omega_n\}$, i.e., all vectors of the form $\sum_{n_i\in \ZZ} n_i \omega_i$, is called the lattice of algebraically integral weights. Since $G$ is a simply-connected semisimple Lie group, the analytically integral weights coincide with the algebraically integral weights (see \cite[Chapter IV.7]{knapp2013lie}). Thus, each algebraically integral weight induces a character of the maximal torus $T$.  
 Let $S$ again be the center of the real stabilizer group $R$, and let $\sff,\sff^*$ denote the Lie algebra of $S$ and its dual space, respectively. More precisely,  
\begin{align*}
\sff=\{H\in \tf : \alpha (H)=0, \ \forall \alpha \in \Delta(\tf,\rf) \},\quad \sff^*=\{\beta\in \tf^* : (\alpha,\beta )=0, \ \forall \alpha \in \Delta(\tf,\rf) \}.
\end{align*}
Then $(\sff^\CC)^*=\CC \langle \omega_1,\ldots,\omega_k\rangle$ and the intersection of the weight lattice with $(\sff^\CC)^*$ is $\ZZ\langle \omega_1,\ldots, \omega_k\rangle$. We call the elements of this sublattice \emph{integral weights on $X$}. Given such an element $\lambda$, there exists an associated character $\chi^\lambda: S^\CC\rightarrow \CC^*$ defined by 
\begin{align}\label{char1}
\chi^\lambda(\exp{v})=\exp{ \lambda(v)}
\end{align}
for all $v \in \mathfrak{s}^\CC \subset \mathfrak{t}^\CC$. The character $\chi^\lambda$ can be extended to $P$ in the following way. The Lie subalgebra $\pf$ can be decomposed as follows:
\begin{align*}
\pf &
=\underbrace{\sff^\CC\oplus \sum_{\beta\in D^+} \gf_{\beta}}_{\pf_1} \oplus  \underbrace{\sum_{\alpha\in \Pi'} \CC\langle H_\alpha \rangle \oplus \sum_{\alpha\in  \Pi' } (\gf_{\alpha}\oplus \gf_{-\alpha})}_{\pf_2}\\
&= \sff^\CC \oplus \nf\ \oplus \pf_2,
\end{align*} 
where $ \pf_1$ is the solvable part of $\pf$, $\nf $ is the nilpotent part of $\pf_1$, and $\pf_2$ is the semisimple part of $\pf$. Each integral weight $\lambda$ on $X$ can be extended to a complex Lie algebra homomorphism $\sigma:\pf \rightarrow \CC$ by defining the extension to be $0$ on both $\nf$ and $\pf_2$. Then the corresponding holomorphic character $\chi^\sigma$ extends $\chi^\lambda$ from $S^\CC$ to $P$. By abuse of notation, we will denote this extension by $\chi^\lambda$. The following proposition shows that all characters of $P$ arise by extension in this way.

\begin{pro}\label{classlinebundle1}
For all holomorphic characters $\chi: P\rightarrow \CC^*$ there exists an integral weight $\lambda$ on $X$ such that $\chi=\chi^\lambda$.
\end{pro}
\begin{proof}
The character $\chi: P\rightarrow \CC^*$ induces a Lie algebra homomorphism $\sigma: \pf\rightarrow \CC$. It suffices to prove that $\sigma$ is trivial if restricted to $\nf$ and $\pf_2$. Notice that $\pf_1$ is a solvable Lie algebra and $\nf=[\pf_1,\pf_1]$.  The restriction of $\sigma$ to $\nf$ must be trivial as $\CC$ is an abelian Lie algebra. Since $\pf_2$ is a semisimple Lie algebra, for each root $\alpha \in D^+$ there exist $X_{\alpha}\in \gf_\alpha$, $X_{-\alpha}\in \gf_{-\alpha}$ and $H_{\alpha}=[X_\alpha,X_{-\alpha}]$ such that 
\begin{align*}
\mathfrak{l}_{\alpha}:=\CC\langle H_\alpha,X_{\alpha},X_{-\alpha}\rangle\cong \sff\lf(2,\CC).
\end{align*}
By the representation theory of $\sff\lf(2,\CC)$, the only possible $1$-dimensional representation of $\mathfrak{l}_\alpha$ is trivial. Since the restriction of $\sigma$ to each $\mathfrak{l}_\alpha$ is trivial, we conclude that $\sigma|_{\pf_2}=0$. 
\end{proof}

Summarizing, we have now proved that every integral weight $\lambda\in \ZZ\langle \omega_1,\ldots, \omega_k \rangle$ induces a character, $\chi^\lambda: P\rightarrow \CC^*$, hence a homogeneous line bundle $L_\lambda$. Conversely, every holomorphic line bundle $L$ is of the form $L\cong L_\lambda$ for some $\lambda \in \ZZ\langle \omega_1, \ldots, \omega_k \rangle$. 
Recall that an integral weight $\lambda$ on $X$ is called  \textit{dominant} if $\lambda=\sum_{i=1}^kn_i\omega_i$ with $n_i>0$.
By the Highest Weight Theorem, for each dominant integral weight $\lambda$ there exists a unique finite-dimensional irreducible complex representation $V(\lambda)$ of $G$ with highest weight $\lambda$. The Bott-Borel-Weil Theorem \cite[Section 7]{bott1957homogeneous} states that the space of global sections of $L_\lambda$ is isomorphic to $V(\lambda)$ as a $G$-module. 
Also by the Bott-Borel-Weil Theorem, the global sections of $L_\lambda$ induce an embedding $X\hookrightarrow \PP(V(\lambda)^*)$, so $L_\lambda$ is very ample. In conclusion, we have the following.

\begin{thm} \label{classlinebundle2}
The Picard group of the compact homogeneous K\"ahler manifold $X$ can be identified with the sublattice $\ZZ\langle\omega_1,\ldots,\omega_k\rangle$ of the lattice of integral weights. Under this identification, ample line bundles correspond to dominant integral weights, and are automatically very ample.
\end{thm}

\begin{proof}
Given the previous discussion, we only need to prove that if $L_\lambda$ is ample, then the weight $\lambda$ is dominant. In Section \ref{2f0}, we will calculate the curvature form of $L_\lambda$ in (\ref{chern2}). The curvature form is positive if and only if $(\lambda, \alpha)>0$ for all $\alpha \in D^+$,
which implies that $\lambda$ is dominant. 
\end{proof}

\subsection{Invariant $2$-forms on $X$} \label{2f0}

In this section, we summarize the results of $G$-invariant K\"ahler 2-forms on $X$. This is needed for the proof of Theorem B in Section 4.2 below. Afterwards, we compute some important invariant K\"ahler forms on $X$, such as the related Chern forms of line bundles and the K\"ahler-Einstein form on $X$. These computations were used in the proof of Theorem 2.5 above.

Throughout this section, let $X$ be a general homogeneous manifold with respect to $(G,R)$. Fixing a distinguished point $p\in X$, let $\{X_\alpha, Y_\alpha, \alpha\in D^+\}$ be a normal adapted basis of $T_p X$ as defined in (\ref{rootvector1}), (\ref{rootvector2}). Then, we write $\{\eta_\alpha, \xi_\alpha, \alpha \in D^+ \}$ as its dual basis of $T_p^* X$. 

The results of $G$-invariant K\"ahler 2-forms on $X$ mainly comes from (\cite{besse2007einstein}, chapter 8). A slight difference is that we attempt to rewrite the theory without embedding $X$ into $\gf$ as an orbit of adjoint representation but in terms of basis $\{\eta_\alpha,\xi_\alpha; \alpha\in D^+\}$ at distinguished point $p$. Recall that the tangent vectors at $p$ can identify with some element in $\gf$ as (\ref{vector1}). All tangent vectors at each point on $X$ can be generated in this way. Globally, this infinitesimal transformation generate a vector field (may not $G$-invariant). If $A\in \gf$, we write $V_A$ as a vector field on $X$ defined by
\begin{align*}
V_A(x)=\frac{d}{dt}\Big|_{t=0}\exp(tA)(x), \qquad \text{ for all } x\in X.	.
\end{align*}
We call $V_A$ the \textit{fundamental vector field} related to $A$. Then, by quick calculation, the Lie bracket of two fundamental vector fields is also  fundamental as
\begin{align}\label{vector2}
[V_A, V_B]=-V_{[A,B]},
\end{align}
The reason we introduce fundamental vector fields is to test invariant 2-forms on $X$. Recall that the negative Killing form defines a $G$-invariant metric $(\cdot,\cdot)$ on $\gf$. Let $S\in \sff$, a 2-form can be defined at distinguished point $p$ as follows,
\begin{align*}
\omega_S(V_A, V_B)=(S,  [A,B]).
\end{align*}  
Noting that $\sff$ is the center of $\rf$, $\omega_S$ is well-defined. It is easy to see $\omega_S$ is $R$-invariant. By taking $N \in \rf$ and using (\ref{vector2}), we have
\begin{align*}
L_{V_N} \omega_S (V_A, V_B)= 
&- \omega_S ([V_N, V_A],V_B) - \omega_S (V_A, [V_N,V_B])\\
=&\ \omega_S(V_{[N,A]},V_B) + \omega_S ( V_A, V_{[N, B]})\\
=&\ (S, [[N,A],B])+(S, [A,[N,B]])\\
=&\ ([S,N], [A,B])=0.
\end{align*} 
Hence, we can define an invariant global 2-form by $G$-action on $X$. Then, we can check that the global 2-form, also denoted by $\omega_S$ is a closed, real, $(1,1)$ form on $X$. The closedness directly follows from the following calculation,
\begin{align*}
d\omega_S (V_A,V_B,V_C)= &
V_A (\omega_S(V_B, V_C))+V_B(\omega_S (V_C,V_A))+V_C (\omega_S(V_A,V_B))\\
& \qquad -\omega_S([V_A,V_B], V_C) - \omega_S([V_B, V_C],V_A) - \omega_S([V_C,V_A], V_B)\\
=& \omega_S([V_A,V_B],V_C)+\omega_s([V_B,V_C],V_A)+\omega_S([V_C,V_A],V_A)=0,
\end{align*}
where the second equality is due to $G$-invariance and the last one is based on Jacobi identity. $\omega_S$ is $(1,1)$-form by testing $\omega_S$ on basis of tangent vector space, $\{ X_\alpha, Y_\alpha,\alpha\in D^+ \}$. It is observed that the only nonvanishing terms are 
\begin{align} \label{2f1}
\omega_S(X_\alpha, Y_\alpha)=(S, [X_\alpha, Y_\alpha])= (S, - H_{\alpha})=-\frac{1}{2i}\alpha(S), \qquad \text{ for all } \alpha \in D^+ 
\end{align} 
Hence, $\omega_S$ is of type $(1,1)$. The previous observation (\ref{2f1}) indicates that we can write down the $\omega_S$ explicitly in terms of covector basis $\{\eta_\alpha,\xi_\alpha, \alpha\in D^+\}$ at $p\in X$. Precisely,
\begin{align} \label{2f2}
\omega_s=\sum_{\alpha\in D^+} C_\alpha(S) \eta_\alpha\wedge \xi_\alpha,\qquad C_\alpha(S)=-\frac{1}{2i}{\alpha(S)}
\end{align}
where $C_\alpha$ can be viewed as a linear function in $\sff^*$. Also by $G$-action, we obtain a global 2-form, as well, denoted by
$\omega_S =\sum C_\alpha(S) \eta_\alpha \wedge \xi_\alpha$. An intriguing point is that $\eta_\alpha$, $\xi_\alpha$ are usually not well-defined as a global invariant 1-form. 

Conversely, given a $G$-invariant closed real $(1,1)$ form, $\omega$, it can always be written as formula (\ref{2f2}) related to some $S$. Since $\omega$ is $G$-invariant, for each $A,B,N\in \gf$,
\begin{align*}
0 = (L_{V_N}\omega)(V_A, V_B)
= L_{V_N}(\omega(V_A, V_B)) - \omega([V_N,V_A],V_B) - \omega(V_A, [V_N,V_B]) 
\end{align*} 
Especially, by taking $N$ to be an element in $\tf$, $A,B$ to be elements of $\{X_\alpha, Y_\alpha, \alpha \in D^+\}$, noting that $V_N(\omega(V_A, V_B))$ is vanishing here, we have
\begin{enumerate}
\item[$\bullet$] The only nonvanishing case are $\displaystyle\omega(X_\alpha,Y_\alpha)=C_\alpha,$ $\alpha\in D^+$.
\item[$\bullet$] All other cases are vanishing. Precisely, for two different roots in $D^+$, $\alpha, \beta$, $\omega(X_\alpha, X_\beta)=\omega(X_\alpha, Y_\beta)=\omega(Y_\alpha,Y_\beta)=0.$
\end{enumerate}
It suffices to show that $C_\alpha \in \sff^*$. Since $2$-form $\omega$ and inner product in $\gf$ can be extended linearly to complex field, consider the following complex vectors
\begin{align} \label{cxvector}
U=X_\alpha-iY_\alpha, \quad V=X_\beta-iY_\beta,\quad W=X_{\alpha+\beta}+iY_{\alpha+\beta},
\end{align}
where $\alpha, \beta, \alpha+\beta\in D^+$ and $X_{\alpha,\beta,\alpha+\beta}$, $Y_{\alpha,\beta,\alpha+\beta}$ are taking from the normal adapted basis as in (\ref{rootvector1}), (\ref{rootvector2}). Viewing $U$, $V$, $W$ as elements in $\gf\otimes \CC$,  the $G$-invariant inner product satisfies,
\begin{align*}
(W,[U,V])=([W,U],V).	
\end{align*}
Therefore, it is easy to check $U,V,W$ satisfies the following , 
\begin{align*}
[U,V]=\lambda \overline{W},
\quad [V,W]=\lambda \overline{U},
\quad [W,U]=\lambda \overline{V}	
\end{align*}
By closedness and $G$-invariance of $\omega$, we have
\begin{align}
d\omega (U, V, W)	= 
&\ \omega([U, V], W)+\omega([V, W], U)+\omega([W, U], V) \nonumber \\
= &- \lambda\big(\omega(\overline{W}, W)+\omega(\overline{U}, U)+\omega(\overline{V}, V)\big)=0. \label{2f3}
\end{align}
Inserting (\ref{cxvector}) into (\ref{2f3}), we have
\begin{align*}
\omega(X_{\alpha+\beta}, Y_{\alpha+\beta}) = 
\omega(X_\alpha,Y_\alpha) +\omega(X_\beta, Y_\beta),	
\end{align*}
hence, $C_{\alpha+\beta}=C_{\alpha}+C_\beta\ (*)$. Noticing that $C$ defines a function on $D^+$, the condition $(*)$ implies that $C$ can be extended to a linear function on $\sff^*$ as $D^+$ generates $\sff^*$. In other words,
there exists an element $S\in \sff$ such that 
\begin{align*}
C_\alpha=-\frac{1}{2i}\alpha(S)	
\end{align*}
In conclusion, we have the following proposition,
\begin{pro} \label{2f4}
For any $S\in \sff$, the related two form $\omega_S$ is a $G$-invariant  closed real $(1,1)$-form. Conversely, any $G$-invariant  closed real $(1,1)$ form, $\omega$, can be related to a unique $S\in\sff$, i.e. $\omega=\omega_S$. Furthermore, at a distinguished point $p\in X$, $\omega_S$ is written explicitly as (\ref{2f2}). $\omega_S$ is positive if and only if $C_\alpha (S)$ is positive for all $\alpha\in D^+$.
\end{pro}

Proposition \ref{2f4} implies a bijection from $\sff$ to the space of $G$-invariant closed real $(1,1)$ forms on $X$, $S\rightarrow \omega_S$. There is a special element in $G$-invariant closed real $(1,1)$ forms, K\"ahler-Einstein form, and the related element in $\sff$ is given by
 \begin{align*}
S_{KE}=2\sum_{\alpha\in D^+} H_{\alpha}.
\end{align*}
So we can write $\omega_{KE}$ explicitly at distinguished point $p\in X$, with respect to the normal adapted basis. Let $\delta$ be the sum of all positive roots in $D^+$
\begin{align} \label{KEform}
	\omega_{KE}=\frac{1}{2}\sum_{\alpha \in D^+} (\alpha,\delta)\eta_\alpha \wedge \xi_\alpha.
\end{align}
The Ricci curvature with respect to $\omega_{KE}$ can be represented explicitly
\begin{align} \label{ricci}
\rho_X = \frac{1}{2} \sum_{\alpha \in D+} (\alpha,\delta) \eta_\alpha \wedge \xi_\alpha=\omega_{KE}.
\end{align}
The positivity of $\omega_{KE}$ and the calculation of Ricci form (\ref{ricci}) can be found in (\cite{besse2007einstein}, section 8C). Hence, each homogenous space is Fano.

Fixing a line bundle related to a weight $\lambda$, the Chern class of $L_\lambda$ can be represented by the form $-i \partial\overline{\partial} \log r^2$, where $r$ is a radian function induced by some hermitian metric on $L_\lambda$. Then, we have
\begin{align}\label{chern1}
\frac{1}{2}dd^c\log r^2 
= -d\big(J\frac{dr}{r}\big).
\end{align}
Indeed, the 1-form $
\displaystyle -Jdr/r$ is induced by circle action along each fibre and more details will be discussed in the next section. Referring to the Proposition \ref{der1}, the curvature form can be represented as
\begin{align}\label{chern2}
		-\frac{1}{2}dd^c \log h &= \frac{1}{2}\sum_{\alpha\in\Delta^+}  (\lambda,\alpha) \eta_\alpha\wedge \xi_\alpha.
\end{align}
Comparing with (\ref{KEform}), the K\"ahler-Einstein metric is the curvature form associated with the anti-canonical bundle, $L_{\delta}$.

\section{Invariant $dd^C$-lemmas on homogeneous line bundles}\label{sec:ddbar}

This section is dedicated to the proof of Theorem \ref{kf1}. Section \ref{lefinv} discusses invariant $1$-forms on the unit circle bundle, $M$, of a homogeneous line bundle $L$ over a compact homogeneous K\"ahler manifold $X$. Section \ref{ss:diffof1form} calculates the differentials of the invariant 1-forms given in Section \ref{lefinv}. Assuming that $L$ is negative, Section \ref{ss:invddbar} proves an invariant $dd^C$-lemma, and we also show that this $dd^C$-lemma can be false if $L$ is not negative. Finally, Section \ref{ss:mainproof} combines these results to prove Theorem \ref{kf1}. 

\subsection{Invariant $1$-forms on $M$}
\label{lefinv}

Let $X$ be a homogeneous compact K\"ahler manifold as before, and let $L=L_\lambda$ be a homogeneous line bundle over $X$ given by an integral weight $\lambda$. Given the data $(X,L)$, in this subsection, we will determine all left-invariant 1-forms on the unit circle bundle $M$ of $X$. To the best of our knowledge, the main result here (Proposition \ref{leftvf1}) is new.

Recall the $G$-action on $L$ defined by restricting the $G^\CC$-action of (\ref{Gaction2}) to $G \subset G^\CC$. For each homogeneous line bundle $L$, there is a natural $G$-invariant hermitian metric $h$ induced by the standard hermitian metric in $\CC$. In particular, according to construction of homogeneous line bundle in (\ref{construcL}), let $q_0= \overline{(g,z)} \in L_\lambda$ for $(g,z)\in G^\CC \times \CC$, then 
$h(q_0,q_0)=|z|^2$.
The hermitian metric $h$ induces a radian function $r$ on $L$. Then, $G$ acts transitively on each level set of $r$, an $S^1$ bundle of $X$, denoted as $M(r)$. Away from zero level set, there is a canonical invariant vector field, $\partial/\partial r$ on $L$ pointing in radius direction. We shall find the set of all invariant vector fields on each level set $M(r)$. Notice that, in general, left-invariant vector fields on $G$ are not always well-defined over $M(r)$, as the left action by the stabilizer group on the tangent space at one point can be nontrivial. Let $M=M(1)$ and let $\T_M$ be the space of all global $G$-invariant vector fields over $M$. $\T_M$ always contains one element, $X_0=J(\partial /\partial r)$, generating a circle action on each fiber. The other elements of $\T_M$ strongly depend on the base manifold $X$ and the integral weight $\lambda$. According to
(\ref{tanspace1}), the tangent space at a distinguished point $p\in X$ can be identified with a subspace of $\gf$. Then, we can choose a normal adapted basis $\{X_\alpha,Y_\alpha\}_{\alpha \in D^+}$ as in (\ref{rootvector1}), (\ref{rootvector2}).

Based on the choice of $\{H_\alpha, X_\alpha, Y_\alpha\}$ as in (\ref{rootvector1}), (\ref{rootvector2}), one can easily check the following Lie algebra structure.
\begin{align*}
	&[H_\alpha, X_\alpha]=-\frac{|\alpha|^2}{2} Y_\alpha \\
	&[H_\alpha, Y_\alpha]= \frac{|\alpha|^2}{2} X_\alpha \\
	&[X_\alpha, Y_\alpha]= - H_\alpha
\end{align*}
Let $q$ be a distinguished point in $M$, then the tangent space at $q$ can be identified with $\RR\langle X_\alpha, Y_\alpha, \alpha\in D^+ \rangle\oplus \RR X_0$ in the following sense. Consider the $G$-equivariant bundle projections,
\begin{align*}
	\xymatrix{G\ar^{\tilde{\pi}}[r]\ar[dr]_\pi&\ M\ar[d]^{\pi |_M}\\ & X}
\end{align*}
with $\tilde{\pi}(e)=q\in M$. At the distinguished point $q\in M$ with $\pi(q)=p\in X$, assume that the stabilizer group at $p\in X$ is $R$ and the stabilizer group at $q\in M$ is $R_0$. By $G$-action on $L$, $g\in R_0$ if and only if,
$$g(q)= g(e, \theta)= (e, \chi^\lambda(g) \theta),$$ 
which implies that $R_0= \ker \chi^\lambda : R\rightarrow S^1 \subset \CC^*$ These projections induce the mapping on tangent spaces $\tilde{\pi}_*:\gf\rightarrow T_qM$ by
\begin{align}\label{vector1}
N\in \gf \mapsto \frac{d}{dt}\Big|_{t=0}\tilde{\pi}\circ\exp(tN).
\end{align}
Similarly, we can define the mapping $\pi_*:\gf\rightarrow T_p X$. By abusing notation, we write $X_\alpha, Y_\alpha$ instead of $\tilde{\pi}_*(X_\alpha)$, $\tilde{\pi}_*(Y_\alpha)$ and $\pi_*(X_\alpha),\ \pi_*(Y_\alpha)$ as $X_\alpha$, $Y_\alpha$. Notice that the left-invariant vector fields on $G$ can be identified with $\gf$. To determine the space of invariant vector fields, $\T_M$, on $M_1$, we observe that for any $R_0$-invariant vector $v\in T_q M$ and $g_1(q)=g_2(q)$, then $g_1=g_2 r$ with $r\in R_0$ and 
\begin{align*}
(g_1)_*(v)=(g_2 r)_*(v)=(g_2)_* r_*(v)=(g_2)_*(v),
\end{align*}
So each $R_0$-invariant vector of $T_q M$ determines a left-invariant vector fields on $M$. There is an one-to-one correspondence between $\T_M$ and the $R_0$-invariant space of $T_q M$. In particular, in the case that $X=G/T$ with $T$ a maximal torus of $G$, we have the following proposition	

\begin{pro}\label{leftvf1}
Let $X$ be a compact  K\"ahler manifold homogeneous under $G$ with stabilizer group $R$, let $\Delta$ be the root system of $(\tf, \gf)$. $D^+$ is defined as in (\ref{tanroot}). Let $M$ be the unit level of homogeneous line bundle determined by an integral weight $\lambda\ne 0$. At a distinguished point $q\in M$, the tangent space $T_q M$ is generated by $\{X_\alpha,Y_\alpha\}_{\alpha\in D^+}$ and $X_0$. Then, there are the following two possiblities for the space of the left-invariant vector fields on $M$,
\begin{enumerate}
	\item[(a)] $\T_M \cong \RR \langle X_0\rangle$
	\item[(b)] $\T_M \cong \RR \langle X_0, X_\alpha, Y_\alpha\rangle$, for some $\alpha \in D^+$
\end{enumerate}
The case \textup{(b)} happens if and only if $\lambda$ is proportional to $\alpha$ and $\alpha+\beta$, $\alpha-\beta$ are not in $\Delta$, for any $\beta\in \Delta(\tf, \rf)$. In particular, if the stabilizer group is a maximal torus, then the case (b) happens if and only if $\lambda$ is proportional to $\alpha$. 
\end{pro}

Before we prove Proposition \ref{leftvf1}, some typical examples will be investigated.

\begin{ex}\label{ex1} Let $X=\CC\PP^1$, $L=\OO(-1)$. In this case, $\OO(-1)\setminus \pi^{-1}(0)$ is biholomorphic to $\CC^2\setminus \{0\}$. The induced $SU(2)$ action on $\CC^2\setminus \{0\}$ coincides with the regular $SU(2)$ representation of $\CC^2$ and the invariant hermitian metric coincides with the standard hermitian metric of $\CC^2$. Therefore, the level set $M$ is isomorphic to $S^2$, then it's easy to find a basis of left-invariant vector fields as follows,
\begin{align*}
X_0&= -y_1\frac{\partial}{\partial x_1}+x_1 \frac{\partial}{\partial y_1}-y_2 \frac{\partial}{\partial x_2}+x_2\frac{\partial}{\partial y_2}  \\
X&=\ \ x_2\frac{\partial}{\partial x_1}- y_2\frac{\partial}{\partial y_1}-x_1 \frac{\partial}{\partial x_2}+ y_1\frac{\partial}{\partial y_2}\\
Y&=	-y_2\frac{\partial}{\partial x_1}-x_2 \frac{\partial}{\partial y_1}+y_1 \frac{\partial}{\partial x_2}+x_1\frac{\partial}{\partial y_2}.
\end{align*}
After removing the zero level, $\OO(-1)$ can be viewed as a covering space of $\OO(-n)$. Noting that the frame $\{X_0, X, Y\}$ is $S^1$-invariant, $\{X_0,X,Y\}$ induces a left-invariant frame on $\OO(-n)$. Thus, by coincidence, for negative line bundles over $\CC\PP^1$, their left-invariant vector fields can span the whole tangent space at each point of the level set $M$. This result does not hold in any other cases. We can understand this example in terms of Proposition \ref{leftvf1}. There is only one simple root for $SU(2)$ and all negative integral weights are proportional to the simple root. Hence, we have $\T_M \cong \RR\langle X_0, X, Y\rangle$
\end{ex}

\begin{ex} \label{exinvv2} Let $X=SU(3)/T^{2} $. Recall the basic notions of semisimple Lie group $SU(n)$. The Lie algebra $\suf(n)$ is the set of trace zero skew-hermitian matrices of order $n$. A Cartan sub-algebra $\tf$ is the Lie algebra of diagonal matrices in $\suf(N)$. In particular, in the case of $SU(3)$, A Cartan sub-algebra is generated by $t_1=\diag(i,0,-i)$ and $t_2=\diag(0,i,-i)$. The set of positive roots with respect to $\tf$ consists of three elements $\{\alpha,\beta,\gamma\}$, and the normal adapted basis is given as follows: 
\begin{align*}
	X_\alpha=C\begin{pmatrix}
	 0 & 1 & 0 \\
	 -1 &  0 & 0 \\
	 0  &  0  & 0 
     \end{pmatrix}, \qquad 
   Y_\alpha=C\begin{pmatrix} 0 & i & 0 \\ i & 0 &  0\\ 0 & 0 & 0
\end{pmatrix},\\
X_\beta=C\begin{pmatrix}
	 0 &  0 & 1 \\
	 0 &  0 & 0 \\
	 -1 &  0  & 0 
     \end{pmatrix}, \qquad 
   Y_\beta=C\begin{pmatrix} 0 & 0 & i \\ 0 & 0 & 0 \\i & 0 & 0
\end{pmatrix},\\
X_\gamma=C\begin{pmatrix}
	 0 & 0 & 0  \\
	 0 & 0 & 1  \\
	 0 & -1 & 0 
     \end{pmatrix}, \qquad 
   Y_\gamma=C\begin{pmatrix} 0 & 0 & 0 \\ 0 & 0 & i \\ 0 & i & 0
\end{pmatrix},
\end{align*}
where $C$ is a real coefficient to normalize $\ X_{\alpha,\beta,\gamma},\ Y_{\alpha,\beta,\gamma}$.  Consider a distinguished point $q\in M$ with tangent space generated by $\{X_0,\ X_{\alpha,\beta,\gamma},\ Y_{\alpha,\beta,\gamma}\}$. Now, we only focus on the subspace  $V\subset T_q M$ generated by $\{X_{\alpha,\beta, \gamma}, Y_{\alpha, \beta, \gamma}\}$. To simplify the notation in calculation, we introduce a complex coordinate system in $V$; precisely,  $z_{\alpha}=X_\alpha+i Y_{\alpha}$. Let $\sigma_i$ denote the $i$-th element of diagonal matrices. Then, $\alpha,\beta,\gamma$ can be expressed as follows,
\begin{align*}	
&\alpha=\sigma_1- \sigma_2 ,\quad
\beta=\sigma_1- \sigma_3,\quad
\gamma=\sigma_2-\sigma_3.
\end{align*}
Hence, let $\{\alpha,\gamma\}$ be the simple root system of $SU(3)$ and $\beta=\alpha+\gamma$. The fundamental weights are given as follows,
\begin{align*}
	\omega_1=\frac{2}{3}\sigma_1-\frac{1}{3} \sigma_2- \frac{1}{3}\sigma_3,\qquad \omega_2=\frac{1}{3}\sigma_1+\frac{1}{3} \sigma_2 -\frac{2}{3}\sigma_3.
\end{align*}
Let $\OO(p,q)$ denote the line bundle corresponding to the integral weight $\lambda=p\omega_1+q\omega_2$. In the case of $(p,q)\ne 0$. Noting that that kernel of $\lambda= p\omega_1+ q\omega_2$ is $\RR\langle q t_1 - (p+q) t_2\rangle$. Then, the $S^1$ action is given by $T_\theta^{p,q}=\diag(e^{iq \theta}, e^{-i(p+q)\theta},e^{ip\theta})$ and if we represent the action on complex coordinate system $(z_\alpha, z_\beta, z_\gamma)$, we have
\begin{align*}
T^{p,q}_\theta(z_\alpha,z_\beta,z_\gamma)=(e^{-i(p+2q)\theta}z_\alpha, e^{i(p-q)\theta} z_\beta, e^{i(2p+q)\theta}z_\gamma ).
\end{align*}
In conclusion, we have the following cases 
\begin{align*}
\begin{tabular}{c | c}
\hline
Conditions on $(p,q)$ & \quad Invariant vector fields over $M$ \quad\\
\hline
$p=-2q $ & $X_0,\quad  X_\alpha, \quad Y_\alpha$\\
\hline
$p=q $ & $X_0, \quad X_\beta, \quad Y_\beta$\\
\hline
$2p=-q $ & $X_0, \quad X_\gamma, \quad Y_\gamma$\\
\hline
Others & $X_0$\\
\hline
\end{tabular}
\end{align*}
Notice that 
\begin{align*}
    \alpha= 2\omega_1-\omega_2, \quad \beta= \omega_1+\omega_2,\quad \gamma= -\omega_1+2\omega_2
\end{align*}
The integral weights in the above table, $-2n\omega_1+ n\omega_2$, $n\omega_1 + n\omega_2$, $-n\omega_1  +2n \omega_2$ are proportional to $\alpha,\beta, \gamma$ respectively in agreement with Proposition \ref{leftvf1}.
\end{ex}

\begin{ex} Let $X=SU(4)/R$ where $R$ consists of all elements of the following type,
\begin{align*}
\begin{pmatrix}
	 * & * & 0 & 0 \\
	 * & * & 0 & 0 \\
	 0 & 0 & * & 0 \\
	 0 & 0 & 0 & *
     \end{pmatrix}.
\end{align*}
Let $\sigma_i$ denote the $i$-th element of matrices. Then the positive roots are given as 
\begin{align*}
\Delta^+=\{\sigma_i-\sigma_j,\ i < j\}.
\end{align*}
In particular, the simple roots are,
\begin{align*}
\lambda_i= \sigma_i - \sigma_{i+1},\qquad i=1,2,3.
\end{align*}
Then, the fundamental weights corresponding to $\{\lambda_i\}_{i=1,2,3}$ are the following, 
\begin{align*}
\omega_1= \frac{3}{4} \sigma_1 - \frac{1}{4} \sigma_2 -\frac{1}{4}\sigma_3 - \frac{1}{4} \sigma_4,\\
\omega_2= \frac{1}{2} \sigma_1 + \frac{1}{2} \sigma_2 - \frac{1}{2}\sigma_3 - \frac{1}{2} \sigma_4,\\
\omega_3= \frac{1}{4} \sigma_1 + \frac{1}{4} \sigma_2 + \frac{1}{4} \sigma_3 - \frac{3}{4} \sigma_4.
\end{align*}
According to theorem \ref{classlinebundle2}, each holomorphic line bundles over $X$ is related to a integral weight $\displaystyle\lambda(n,m)= \sum n\omega_1 + m\omega_2$. Let $L_{n,m}$ denote the line bundle related to $\lambda(n,m)$ and $M$, the circle bundle of $X$ induced by $L_{n,m}$.  Let $q$ be a distinguished point in $M$. If we write $X_{ij}, Y_{ij}$ as the normal adapted basis of root space of $\sigma_i-\sigma_j$, then the tangent space at $q$ is generated by $\{X_0, X_{13,14,23,24,34}, Y_{13,14,23,24,34}\}$. Let $z_{ij}= X_{ij} + iY_{ij}$. Notice that the circle $T_{12}(\theta)=\diag(e^{i\theta}, e^{-i\theta},1,1)$ is always in the kernel of $\chi^{\lambda(n,m)}$ for all $n,m\in \ZZ$; hence the circle stabilizes the point $q\in M$. One can check the $T_{12}(\theta)$-action on $T_{q}M$,
\begin{align*}
T_{12}(\theta)(z_{13},z_{14}, z_{23}, z_{24}, z_{34}) = (e^{i\theta} z_{13}, e^{i\theta}z_{14}, e^{-i\theta}z_{23}, e^{-i\theta}z_{24}, z_{34}).
\end{align*} 
Assume $m=-2n$, then the kernel of $\chi^\lambda$ is $R_0= CS$, where $C$ is the center of $R_0$ and can be represented as $T'(\theta)=\diag(e^{i\theta}, e^{i\theta}, e^{-i\theta}, e^{-i\theta})$. $S$ is the semisimple part of $R_0$, which is isomorphic to $SU(2)$. In particular,
\begin{align*}
S\cong \begin{pmatrix}
SU(2) & 0\\
0 & I_2
\end{pmatrix}
\end{align*}
It is easy to check that $R_0$ acts trivially on $X_{34}$ and $Y_{34}$. When $m\ne -2n$, by similar calculation (only the center of $R_0$ changes), we can show that there is no nontrivial invariant vectors in $T_q M$.
 
In conclusion, we have
\begin{align*}
\begin{tabular}{c | c}
\hline
Conditions on $(n,m)$ & \quad Invariant vector fields over $M$ \quad\\
\hline
$m=-2n $ & $X_0,\quad  X_{34}, \quad Y_{34}$\\
\hline
Others & $X_0$\\
\hline
\end{tabular}
\end{align*}
Also notice that $\lambda= n \omega_1 - 2n \omega_2= -n (\sigma_3-\sigma_4) = -n \lambda_3$ and $\lambda_3 \pm \lambda_1$ are not roots, again, in agreement with Proposition  \ref{leftvf1}.
\end{ex}

\begin{proof}[Proof of Proposition \ref{leftvf1}] 
	Let $R_0$ be the stabilizer group at $q\in M$. Notice that $R_0$ is the kernel of character $\chi^\lambda: R\rightarrow S^1\subset \CC$ related to the weight $\lambda$. Let $ V\in T_q M$ be an invariant vector. According to decomposition (\ref{tanspace1}) of $T_q M$, the vector $V$ can be written as
\begin{align*}
	V=\sum_{\alpha\in  D^+}V_\alpha, \qquad V_\alpha\in \RR\langle X_\alpha, Y_\alpha \rangle \text{ and } V_\alpha\ne 0.
\end{align*}	
Notice that $R_0$ preserves each $\RR\langle X_\alpha, Y_\alpha \rangle$. If a vector $V_\alpha = aX_\alpha+bY_\alpha\in E_{\pm \alpha}$ is invariant under $R_0$ action, then, since $J$ is $R_0$ invariant, which means that $J(aX_\alpha+bY_\alpha)=aY_\alpha-bX_\alpha$ is also $R_0$ invariant. Therefore, all vectors in the space generated by $\langle aX_\alpha+bY_\alpha,-b X_{\alpha}+a Y_\alpha\rangle=\langle X_\alpha, Y_\alpha\rangle = \RR\langle X_\alpha, Y_\alpha \rangle$   are $R_0$ invariant. So we only need to determine the set of all $\alpha \in \Delta^+$ such that $X_\alpha$, $Y_\alpha$ are invariant under $ R_0 = \ker \chi^\lambda$.

 Let $\rf_0=\ker\lambda$, it's easy to see that $\rf_0$ is the Lie algebra of the stabilizer group $R_0$. Let $T_0= R_0 \cap T$ associated with Lie algebra $\tf_{0}=\rf_0\cap \tf$. Since $X_\alpha$, $Y_\alpha$ is $T_0$ invariant,
\begin{align}\label{ker1}
	[H_0,X_\alpha]=0=[H_0,Y_\alpha],\qquad \text{for all } H_0\in \tf_0
\end{align}
Noting that
$$[H_0, X_\alpha]= -i\alpha(H_0) Y_\alpha, \qquad [H_0, Y_\alpha] = i \alpha(H_0) X_\alpha. $$
Hence, we have $\ker \alpha= \tf^0 = \ker \lambda|_{\tf}$, hence $\lambda$ is proportional to $\alpha$. 

Let $Z \in \rf_0$, then $X_\alpha,Y_\alpha$ are $R_0$ invariant if and only if 
\begin{align*}
 \tilde{\pi}_*[Z, X_\alpha]=\tilde{\pi}_*[Z,Y_\alpha]=0,	
\end{align*}
Notice that $\rf_0 = \tf_0 \oplus \sum_{\beta \in \Delta^+(\tf,\rf)}\RR\langle X_\beta, Y_\beta \rangle$, where $\tf_0$ is kernel of $\lambda$ restricted in $\tf$. Then,  $X_\alpha$ and $Y_\alpha$ is $R_0$-invariant is equivalent to (\ref{ker1}) and
\begin{align*}
	\tilde{\pi}_*[Z_\beta,X_\alpha]=\tilde{\pi}_*[Z_\beta,Y_\alpha]=0,&\qquad \text{for all } Z_\beta\in \RR\langle X_\beta, Y_\beta \rangle,\ \beta\in \Delta^+(\tf,\rf) \qquad (**)
\end{align*}
According to \cite[Theorem 6.6]{knapp2013lie} and the definition of $\tilde{\pi}_*$,  $(**)$ is equivalent to $\alpha + \beta,\ \alpha -\beta \notin D^+ \cup (-D^+)$. Indead, $\alpha+\beta$ and $\alpha-\beta$ are not roots. Assume that
$$ \alpha +\beta= \gamma \in \Delta(\tf,\rf), $$
then, $\alpha=\gamma - \beta \in \Delta(\tf,\rf)$. But $\alpha \in D^+$, which leads to a contradiction.
\end{proof}

The left-invariant $1$-forms on $M$ can be viewed as the dual space of left-invariant vector fields. More precisely, if we apply the previous notion of adapted basis at $T_q M$, given by $\{X_0, X_\alpha, Y_\alpha; \alpha\in D^+\}$, we write the dual basis of $T^*_qM$ as $\{\eta_0,\eta_\alpha, \xi_\alpha;\alpha\in D^+\}.$ According to Proposition \ref{leftvf1}, the space of left-invariant vector fields is in one-to-one correspondence with the space generated by $\{X_0\}$ or $\{X_0, X_\alpha,Y_\alpha\}$ for some $\alpha \in D^+$. Consider the subset, $\{\eta_0\}$ or $\{\eta_0,\eta_\alpha,\xi_\alpha \}$ of the dual basis, whose elements are $R_0$ invariant. Therefore, $\{\eta_0\}$ or $\{\eta_0,\eta_\alpha,\xi_\alpha \}$ generates the space of $G$-invariant $1$ forms over $M$.

\subsection{The differentials of left invariant 1-forms on $M$} \label{ss:diffof1form}

According to proposition \ref{leftvf1}, the left invariant 1-forms are generated by $\{\eta_0 \}$ or $\{\eta_0, \xi_\alpha, \eta_\alpha \}$. Let us only deal with the second case, because the differential of the first case follows directly from the calculation of the second one. 

Let $M$ be an $S^1$ bundle associated with line bundle $L$. Notice that there is a natural projection $\tilde{\pi}: G\rightarrow M$. If we write $\Omega_1$, $\Omega_2$ as the space of smooth $1$-forms and $2$-forms respectively, then we have the following commutative graph,
\begin{equation}\label{form1}
\begin{aligned}
\xymatrix{\Omega_1(M)\ar[d]_{d}\ar[r]^{\tilde{\pi}^*}& \Omega_1(G)\ar[d]^d\\ \Omega_2(M)\ar[r]^{\tilde{\pi}^*}& \Omega_2(G).
 }	
 \end{aligned}
\end{equation}
Since the left-invariant vector fields are globally generated in $G$, with a natural basis corresponding to $\{H_\alpha,X_\beta,Y_\beta:  \alpha\in\Pi,\beta\in \Delta \}$ and its dual basis $\{h_\alpha,\eta_\beta,\xi_\beta:\alpha\in \Pi,\beta\in \Delta\}$, then the pull back of $\eta_\alpha$, $\xi_\alpha$ in $\Omega_1(M)$ under $\tilde{\pi}$ are exactly $\eta_\alpha$, $\xi_\alpha$ in $\Omega_1(G)$. And the pull-back of $\eta_0$ is a certain combination of $h_\alpha$ determined by weight $\lambda$. Recall the Maurer-Cartan equations on $G$ with respect to the natural basis, 
\begin{align}\label{MC1}
	[H_\alpha,X_\beta]= -\frac{(\alpha,\beta)}{2} Y_\alpha,\quad 
	[H_\alpha,Y_\beta]= \frac{(\alpha,\beta)}{2} X_\alpha,\quad [X_\beta,Y_\beta]= H_\beta.
\end{align}
Besides the above equations, there are also some other nontrivial Lie brackets $[E_{\pm\alpha},E_{\pm \beta}]$, if $\alpha+\beta$ or $\alpha-\beta$ is a root, but those terms are quite messed, we will discuss them separately in the proof. Since Maurer-Cartan equations gives us the derivative of left-invariant $1$-form on $G$, combining the commutative graph (\ref{form1}),  we  have the derivative of left-invariant $1$ form on $M$.

\begin{pro} \label{der1}	Let $X$ be the homogeneous space and $M$, the $S^1$ bundle of $X$ associated with $\lambda$. Assuming that,  at the distinguished point $q\in M$, the space of left-invariant $1$-forms 
can be identified with the spaces generated by $\{\eta_0,\eta_\alpha,\xi_\alpha\}$, $\alpha \in D^+$. In this case, $\lambda$ is proportional to $\alpha$, assuming that $\alpha=-l\lambda$. then the derivative are given by
\begin{align}\label{form2}
	d\eta_0 &= - \frac{1}{2}\sum_{\alpha\in D^+}  (\lambda,\alpha) \eta_\alpha\wedge \xi_\alpha, \\
	 d\eta_\alpha &= -l \eta_0 \wedge \xi_\alpha-\frac{C^\alpha_{\beta,-\gamma}}{2}\sum_{\substack{\beta,\gamma\in D^+\\ \beta-\gamma=\alpha}} (  \eta_\beta\wedge \eta_\gamma+\xi_\beta\wedge \xi_\gamma )\nonumber
	\\& \qquad\qquad\qquad\qquad\qquad\qquad \label{form3}  -\frac{C^\alpha_{\beta,\gamma}}{2}\sum_{\substack{\beta,\gamma\in D^+\\ \beta+\gamma=\alpha}}(\eta_\beta\wedge\eta_\gamma-\xi_\beta\wedge\xi_{\gamma})
\end{align}
\begin{align}
	 d\xi_\alpha &= l \eta_0 \wedge \eta_\alpha + \frac{C^\alpha_{\beta,-\gamma}}{2}\sum_{\substack{\beta,\gamma\in D^+
	 \\ \beta-\gamma=\alpha}}(\eta_\beta\wedge \xi_\gamma-\xi_\beta \wedge \eta_\gamma )\nonumber	\\& \qquad\qquad\qquad\qquad\qquad\qquad \label{form4} -\frac{C^\alpha_{\beta,\gamma}}{2}\sum_{\substack{\beta,\gamma\in D^+\\ \beta+\gamma=\alpha}}(\eta_\beta\wedge\xi_\gamma+\xi_\beta\wedge\eta_{\gamma}).
\end{align}
where the coefficients $C^\alpha_{\beta,-\gamma}$ are the coefficients from Maureur-Cartan equation. Let $E_\alpha$, $E_\beta$, $E_{-\gamma}$ be the root vector of $\alpha$, $\beta$, $-\gamma$ satisfying (\ref{rootvector1}), then 
\begin{align*}
	[E_\beta,E_{-\gamma}]=C^\alpha_{\beta,-\gamma} E_{\alpha}
\end{align*}
\end{pro}

\begin{proof}
Let $\{\omega_1, \cdots, \omega_k\}$ be fundamental integral weights on $X$ and $\lambda=\sum_{i}n_i\omega_i$. By definition of fundamental weights and our setting of $H_{\alpha_i}$,  we can evaluate $H_{\alpha_i}$ under $\lambda$.
\begin{align*}
\lambda(H_{\alpha_i})=\sum_{j}n_j\omega_j(H_{\alpha_i})
=\frac{1}{2i}\sum_jn_j(\omega_j,\alpha_i)&\\
=-i\frac{|\alpha_i|^2}{4}&\sum_j\frac{2n_j(\omega_j,\alpha_i)}{(\alpha_i,\alpha_i)}
=-i\frac{|\alpha_i|^2}{4} n_i
\end{align*}
By the definition of $L_\lambda$, at distinguished point $q\in M$ with $\tilde{\pi}(e)=q$, we have
\begin{align}\label{ker2}
	\tilde{\pi}_*(H_{\alpha_i})
	=\frac{d}{dt}\Big|_{t=0}\chi^\lambda(\exp (t H_{\alpha_{i}}))
	=\frac{d}{dt}\Big|_{t=0}\exp (t \lambda(H_{\alpha_{i}}) )=-\frac{|\alpha_i|^2}{4}n_i X_0.
\end{align}
For each $i$, 
\begin{align*}
	\tilde{\pi}^*(\eta_0)(H_{\alpha_i})=\eta_0(\tilde{\pi}_*(H_{\alpha_i}))=-\frac{|\alpha_i|^2}{4}n_i \eta_0(X_0)=-\frac{|\alpha_i|^2}{4}n_i.
\end{align*}  
Since the previous calculation shows that $\omega_i(H_{\alpha_j})=-i|\alpha_j|^2\delta_{ij}/4$, then the pullback of $\eta_0$ can be represented by $\tilde{\pi}^*\eta_0=-i\lambda$. To get the formula (\ref{form2}), notice that
\begin{align*}
	\tilde{\pi}^*d\eta_0(X_\alpha,Y_\alpha)=
	d\tilde{\pi}^*(\eta_0)(X_\alpha,Y_\alpha)
	&=\tilde{\pi}^*\eta_0([X_\alpha,Y_\alpha])\\
	&=- i \lambda({H_\alpha})
	 =- \frac{1}{2} (\lambda,\alpha).
\end{align*}

To get the formulas (\ref{form3}), (\ref{form4}), if we pull back both sides of formulas by $\tilde{\pi}$, 
\begin{align*}
	\tilde{\pi}^*(d\eta_\alpha)(H_{\alpha_i},Y_\alpha)
	&=-\tilde{\pi}^*\eta_\alpha([H_{\alpha_i},Y_\alpha])
	=-\frac{(\alpha,\alpha_i)}{2},\\
	\frac{2(\alpha,\alpha_i)}{n_i|\alpha_i|^2}\tilde{\pi}^*(\eta_0\wedge\xi_\alpha)(H_{\alpha_i},Y_\alpha)
	&={-i}\frac{2(\alpha, \alpha_i)}{n_i|\alpha_i|^2}\sum_i \lambda\wedge \xi_{\alpha}(H_{\alpha_i},Y_\alpha)
	=-\frac{(\alpha_i, \alpha)}{2}.
\end{align*}
Since $\tilde{\pi}^*$ is injective, $d\eta_\alpha$ admits the term $\displaystyle \frac{2(\alpha,\alpha_i)}{n_i|\alpha_i|^2} \eta_0\wedge \xi_\alpha$. We can show that the coefficient is independent of the index $i$ and related to the factor $l$. Notice that $\lambda$ is proportional to $\alpha$, $\alpha=-l\lambda$, then
\begin{align*}
	\frac{2(\alpha,\alpha_i)}{n_i |\alpha_i|^2}
	=\frac{2(-l\lambda,\alpha_i)}{n_i|\alpha_i|^2}
	=\frac{2(-l\sum_{j}n_j\omega_j,\alpha_i)}{n_i|\alpha_i|^2}=-l.
\end{align*}

To compute the remaining cross terms of $d\eta_\alpha$ and $d\xi_\alpha$, we shall understand the structure of Lie algebra. Notice that $d\eta_\alpha$ has a nonvanishing 2-form related with root $\beta,\ \gamma$ only if 
\begin{enumerate}
\item[(1)] $\alpha=\beta+\gamma$.  
\item[(2)] $\alpha=\beta-\gamma$.	
\end{enumerate}
Both in the case (1) and (2), we argue that $\beta,\gamma\in D^+$. For instance, in the case (1), assume that $\gamma \in \Delta(\tf,\rf)^+$, then we have $\beta= \alpha-\gamma$ is a root, which contradicts the condition in proposition \ref{leftvf1}.

In the case (1), notice that $\displaystyle [E_\beta, E\gamma]= C^\alpha_{\beta, \gamma} E_\alpha $ and $\displaystyle [E_\beta, \overline{E_\gamma}]$ has no $E_\alpha$ terms. Combining with the relation (\ref{rootvector1}), we have,
\begin{align*}
&\eta_\alpha([X_{\beta},X_{\gamma} ])=-\eta_\alpha([Y_\beta,Y_\gamma])=\frac{C^\alpha_{\beta,\gamma}}{2},\\
&\xi_{\alpha}([X_\beta, Y_\gamma])=\xi_{\alpha}([Y_\beta,  X_\gamma])=\frac{C^\alpha_{\beta,\gamma}}{2}.
\end{align*}
The above equation implies that 
\begin{align*}
\tilde{\pi}^*(d\eta_\alpha)(X_\beta, X_\gamma)=
- \pi^*\eta_\alpha([X_\beta,X_\gamma])&= 
-\frac{C^\alpha_{\beta,\gamma}}{2}\\
\tilde{\pi}^*(d\eta_\alpha)(Y_\beta, Y_\gamma)=
-\pi^*\eta_\alpha([Y_\beta,Y_\gamma])&=
\frac{C^\alpha_{\beta,\gamma}}{2}
\end{align*}
Hence, $d\eta_\alpha$ has the term $(C^\alpha_{\beta,\gamma}/2)(-\eta_\beta \wedge \eta_\gamma + \xi_\beta \wedge \xi_\gamma)$. Likewise, we can find the formulas for $d\eta_\alpha$ and $d\xi_\alpha$ as (\ref{form3}) and (\ref{form4}) and we completes the proof.
 \end{proof}

\subsection{Invariant $dd^c$-lemma on $L_\lambda$}\label{ss:invddbar}

Consider the data $(X,L_\lambda)$, a generalized flag variety $X$ of semisimple Lie group $G$ and $L_\lambda$, a nontrivial homogeneous line bundle related to integral weight $\lambda$. According to proposition \ref{leftvf1}, the space of left invariant vector fields on $M$ have two different cases (a) and (b). Suppose that $\lambda$ satisfies one of the following conditions,
\begin{enumerate}
\item[$\bullet$] The space of left invariant vector fields on $M$ satisfies case (a);
\item[$\bullet$] The space of left invariant vector fields on $M$ satisfies case (b). And $\lambda$ is proportional to some positive root $\alpha$ with $\lambda= -l \alpha$, $l>0$. 
\end{enumerate}
Then, we have the following invariant $dd^c$
 lemma.
\begin{pro}\label{ddbar3} Let $(X,L_\lambda)$ satisfy the conditions above. If $\omega$ is an $G$-invariant closed real $(1,1)$-form on $L$, $[\omega]=0$ , then there exists a $G$-invariant K\"ahler potential $\Phi \in \mathcal{C}^\infty (L_\lambda)$ such that.
\begin{equation}
\omega=dd^c \Phi.
\end{equation}
\end{pro}

\begin{proof}
Since $\omega$ is exact, there exists an 1-form $\theta$ such that $d\theta=\omega$. Moreover $\theta$ can choose to be $G$ invariant. Notice that $\omega = d (g^* \theta)$, then, by taking integral over $G$, we obtain a $G$-invariant 1-form $\theta$ with $\omega=d \theta$.

The main idea of proof is to represent $d\theta$ and $\partial\overline{\partial}\phi$ with respect to $G$-invariant coframe, then we can reduce the proof of Proposition \ref{ddbar3} to solving a system of ODE.

Let $L_\lambda$ be the line bundle such that the space of left invariant vector fields on $M$ satisfies case (b). Suppose that $\lambda$ is proportional to $\alpha$. The basis of left-invariant $1$-form on level set $M$ is $\{\eta_0,\eta_\alpha,\xi_\alpha\}$. Then, we can extend the these invariant $1$-forms in radian direction by rescaling $1/r$ on each level $M(r)$. Precisely, There is a natural projection $p: L^\times \rightarrow M$. By identifying $L^\times \cong G\times_{\lambda}\CC^\times$, the projection can be written explicitly,
$$p({(g, z)}) = {(g, |z|^{-1}z)}.$$
Thus, $p^*$ extends $\{\eta_0, \eta_\alpha, \xi_\alpha \}$ to $L^\times$. By abusing the notion, we write $\{\eta_0, \eta_\alpha, \xi_\alpha\}$ as the extended vector fields over $L^{\times}$. Also let $\mu=  dr/r$, hence $\{\mu,\eta_0,\eta_\alpha,\xi_\alpha\}$ forms a basis of invariant $1$-forms over $L^\times$. Then the invariant $1$-form $\theta$ can be represented as
 \begin{align} \label{Form1}
 \theta=\varphi_r(r)\mu+\varphi_0 (r) \eta_0 + \varphi_\alpha(r) \eta_\alpha+ \phi_\alpha(r) \xi_\alpha. 
 \end{align}  

Applying Proposition \ref{der1}, we can take derivative of $\theta$ in (\ref{Form1}),
\begin{align}
	d\theta =& \ r\varphi_0'(r)\mu\wedge\eta_0	
	-\frac{\varphi_0(r)}{2}\sum_{\alpha \in D^+}  (\lambda,\alpha)\eta_\alpha\wedge\xi_\alpha \label{coreform1}\\
	& +r\varphi_\alpha'(r)\mu\wedge \eta_\alpha+ r\phi_\alpha'(r) \mu\wedge \xi_\alpha 
	-{l}\varphi_\alpha  \eta_0\wedge\xi_\alpha
	+{l}\phi_\alpha  \eta_0\wedge\eta_\alpha \label{coreform2}\\
	& -\varphi_\alpha \frac{C^\alpha_{\beta,-\gamma}}{2}\sum(\eta_\beta\wedge\eta_\gamma+\xi_\beta\wedge \xi_\gamma)+\phi_\alpha\frac{C^\alpha_{\beta,-\gamma}}{2} \sum(\eta_\beta\wedge \xi_\gamma-\xi_\beta\wedge \eta_\gamma) \label{coreform3}\\
   & -\varphi_\alpha \frac{C^\alpha_{\beta,\gamma}}{2}\sum(\eta_\beta\wedge\eta_\gamma-\xi_\beta\wedge \xi_\gamma)
   -\phi_\alpha\frac{C^\alpha_{\beta,\gamma}}{2} \sum(\eta_\beta\wedge \xi_\gamma+\xi_\beta\wedge \eta_\gamma) \label{coreform4}
\end{align}
Let $J$ be complex structure. Note that $d\theta$ is real $(1,1)$-form if and only if $Jd\theta=d\theta$. Notice that $J\eta_0=\mu$, $J\xi_\alpha=\eta_\alpha$, then  
\begin{align*}
	Jd\theta=& \ r \varphi'_0(r)\mu\wedge \eta_0 - \frac{\varphi_0(r)}{2} \sum_{\alpha\in D^+}(\lambda,\alpha)\eta_\alpha\wedge\xi_\alpha \\
	&+ r\varphi'_\alpha\eta_0\wedge \xi_\alpha- r\phi'_\alpha \eta_0\wedge \eta_\alpha -{l}\varphi_\alpha\mu\wedge \eta_\alpha-{l}\phi_\alpha\mu\wedge\xi_\alpha\\
	& -\varphi_\alpha \frac{C^\alpha_{\beta,-\gamma}}{2}\sum(\eta_\beta\wedge\eta_\gamma+\xi_\beta\wedge \xi_\gamma)+\phi_\alpha\frac{C^\alpha_{\beta,-\gamma}}{2} \sum(\eta_\beta\wedge \xi_\gamma-\xi_\beta\wedge \eta_\gamma) \\
   & +\varphi_\alpha \frac{C^\alpha_{\beta,\gamma}}{2}\sum(\eta_\beta\wedge\eta_\gamma-\xi_\beta\wedge \xi_\gamma)
   +\phi_\alpha\frac{C^\alpha_{\beta,\gamma}}{2} \sum(\eta_\beta\wedge \xi_\gamma+\xi_\beta\wedge \eta_\gamma) 
\end{align*}
Then, $Jd\theta=d\theta$ implies the following ODE
\begin{align}\label{ODE1}	
r\varphi_\alpha'=-l\varphi_\alpha,\qquad r\phi'_\alpha=-l\phi_\alpha 
\end{align} 
and the terms of line (\ref{coreform4}) are vanishing. To ensure the (\ref{coreform4}) vanishes, 
\begin{align}\label{con1}
C^\alpha_{\beta, \gamma}=0 \qquad \text{ or }\qquad \varphi_\alpha=\phi_\alpha=0.	
\end{align}
Indeed, referring to (\cite{knapp2013lie}, Theorem 6.6), if $\alpha=\beta+\gamma$, for some $\alpha,\beta,\gamma\in \Delta$, then the corresponding constant $C^\alpha_{\beta,\gamma}\ne 0$. Assuming that there exist positive roots, $\beta$ and $\gamma$, satisfying $\alpha=\beta+\gamma$,  by (\ref{con1}), we have $\varphi_\alpha=\phi_\varphi=0$, which automatically satisfies (\ref{ODE1}); hence, if $d\theta$ is a real $(1,1)$ form with some structure constants $C^\alpha_{\beta,\gamma}$ nonvanishing, then $d\theta$ can be written as 
\begin{align*}
	d \theta=
	r\varphi_0'(r)\mu\wedge \eta_0
	-\frac{\varphi_0(r)}{2}\sum_{\alpha\in D^+} (\lambda,\alpha)\eta_\alpha \wedge \xi_\alpha
\end{align*}

When it comes to the cases that the corresponding weight $\lambda$ is proportional to a simple root $\alpha$, i.e. $\alpha$ cannot be written as the sum of two positive roots,  $d\theta$ should satisfy the equation (\ref{ODE1}) and its solution is given by $\varphi=C/r^l$ with $C$ an arbitrary constant. In the sequel, it suffices to show that the constant $C$ in the expression of solution $\varphi$ should equal zero. To prove this, we need to apply the condition that the form, $d\theta$, is well-defined across the zero level of line bundle. Firstly, we take a reference metric $\omega_0$ near the zero level of line bundle $L$. Let $h$ be the canonical invariant hermitian metric defined as before and $r$ be the radial function related to $h$. Also, given a bundle coordinate $u$, we have
\begin{align*}
\omega_\epsilon=
&\pi^*\omega_{KE}+\epsilon\cdot  dd^c r^{{2}}\\	
=& \pi^*\omega_{KE}
+{\epsilon}r^{{2}}\cdot dd^c \log h
+ {\epsilon}h\cdot i du\wedge d\overline{u}.
\end{align*}
It is easy to see that $\omega_\epsilon$ is positive around zero level. To simplify calculation, let $\epsilon$ tends to $0$, then we obtain a semi-positive form $\omega_0=\pi^* \omega_{KE}$. Consider the following integration,
\begin{align}\label{int1}
\int_K \omega_0^{n-1}\wedge d\theta\wedge d\theta,
\end{align}
where the $K$ is a compact neighborhood of zero level, which is defined as $K=\{x\in L, \ r(x) \leq \delta \}$. In the sequel, we write $(\ref{coreform1})$, $(\ref{coreform2})$ and $(\ref{coreform3})$ as $\Theta_1$, $\Theta_2$, $\Theta _3$. Noticing that all crossing terms $\Theta_i\wedge \Theta_j\wedge \omega_0^{n-1}=0$, ($i\ne j$), and $\Theta_3\wedge \Theta_3\wedge \omega_0^{n-1}=0$. Only two nonvanishing terms of the  integration (\ref{int1}) are the following,
\begin{align}\label{int2}
	\int_K\omega_0^{n-1}\wedge d\theta\wedge d\theta 
	=\int_K \omega_0^{n-1} \wedge\Theta_1\wedge \Theta_1
	+\int_K \omega_0^{n-1} \wedge \Theta_2\wedge \Theta_2
\end{align}
Inserting the solution of ODE (\ref{ODE1}), 
\begin{align*}
\Theta_2=
-C_1 l r^{-l} \mu \wedge \eta_\alpha
-C_2 l r^{-l}	 \mu \wedge \xi_\alpha
-C_1 l r^{-l} \eta_0 \wedge \xi_\alpha
+C_2 l r^{-l} \eta_0  \wedge \eta_\alpha,
\end{align*}
then we can compute the two terms in (\ref{int2}) separately,
\begin{align*}
\int_K \Theta_2 \wedge \Theta_2 \wedge \omega_0^{n-1} & =
-\int_K (C_1^2 +C_2^2) l^2 r^{-2l} \mu \wedge \eta_0 \wedge \eta_\alpha \wedge \xi_\alpha \wedge \omega_0^{n-1}\\
& = -(C_1^2 +C_2^2 ) l^2 \int_0^\delta r^{-2l-1} dr \int_{M_1}\eta_0  \wedge \eta_\alpha \wedge \xi_\alpha \wedge \omega_0^{n-1} 
= -\infty
\end{align*}
and assume $\lambda$ is proportional to a positive root by a negative constant, $\sum_{\alpha\in \Delta^+}(\lambda, \alpha)\leq 0$
\begin{align*}
\int_K \Theta_1 \wedge \Theta_1 \wedge \omega_0^{n-1}
&=-\frac{1}{2}\sum_{\alpha\in \Delta^*} \int_K r\varphi_0'(r) \varphi_0(r)  (\lambda,\alpha) \eta_\alpha \wedge \xi_\alpha \wedge \mu \wedge \eta_0 \wedge \omega_0^{n-1} \\
&=-\frac{1}{2}\sum_{\alpha\in \Delta^+}(\lambda, \alpha)\int^{\delta}_0 \varphi_0'(r) \varphi_0(r) dr \int_{M_1} \eta_0 \wedge \eta_\alpha \wedge \xi_\alpha \wedge \omega_0^{n-1}\\
&=C\lim_{\epsilon\rightarrow} (\varphi_0(r))^2\Big|_{\epsilon}^\delta< C', \quad \text{ ($C$, $C'$ are nonnegative constants)} 
\end{align*}
Hence, $\displaystyle, \int_K d\theta \wedge d\theta \wedge \omega_0^{n-1}=-\infty$, which leads to a contradiction. We obtain that $\varphi_\alpha=\phi_\alpha=0$. According to the discussion of two cases, we have $\varphi_\alpha=\phi_{\alpha}=0$ and $d\theta$ can be represented as,
\begin{align} \label{dform1}
d\theta=
r\varphi_0'(r)\mu\wedge \eta_0
- \frac{\varphi_0(r)}{2}\sum_{\alpha\in \Delta^+}(\lambda, \alpha)\eta_\alpha \wedge \xi_\alpha	
\end{align}
Take a function $\Phi(r)$ such that $\Phi'(r)=\varphi_0(r)/r$, we compute $i\partial\overline{\partial} \Phi$,  
\begin{align*}	
	dd^c \Phi
	=-d\cdot Jd \Phi
	=-d\cdot J(\varphi_0(r)\mu)
	=d(\varphi_0(r)\eta_0)
	= r \varphi'_0(r) \mu \wedge \eta_0+\varphi_0 (r) d\eta_0.
\end{align*}
Combining with Proposition \ref{der1}, we obtain
\begin{align}\label{dform2}
dd^c \Phi 
= r \varphi_0'(r)\mu \wedge \eta_0
- \frac{\varphi_0 (r)}{2}\sum_{\alpha \in D^+}(\lambda, \alpha)\eta_\alpha \wedge \xi_\alpha	
\end{align}
Now, we find $\Phi \in C^\infty (L^\times)$ such that $\omega = dd^c \Phi$. It suffices to prove that $\Phi$ can extend smoothly across the zero level of $L$. To prove this, we need to apply a basic fact from complex functions: Let $f: [0,\infty) \rightarrow \RR$ be a smooth function, then,  $g(z)=f(|z|)$ is smooth in $\CC$ if and only if there exits a smooth function $h: [0,\infty) \rightarrow \RR$ such that $f(r)=h(r^2)$. Since $\omega$ is defined in the whole line bundle $L$, on each fibre, the terms of $\omega$,
\begin{align} \label{termsacross0}
r\varphi'_0(r) \mu\wedge\eta_0, \qquad \varphi_0(r) \eta_\alpha\wedge\xi_\alpha
\end{align}
can extend smoothly across the zero level. Notice that $\mu \wedge \eta_0 = Cr^{-2} du \wedge d\overline{u} $, where $u$ denotes the fiber coordinate of $L$. Then, it is easy to see that the terms in (\ref{termsacross0}) is smooth on each fiber if and only if $r^{-1}\varphi'(r)$ and $\varphi(r)$ are smooth on $\CC$ if and only if there exists a smooth function $h: [0,\infty) \rightarrow \RR$ such that $h(r^2)=\varphi(r)$. According to the definition of $\Phi$, we expand $\Phi$ near $0$.
\begin{align*}
\Phi'(r) = h(r^2) /r =C_{-1}/r +C_1 r + C_2 r^3 + \ldots.
\end{align*}
we have 
\begin{align*}
\Phi(r) = C_{-1} \log r + C_0 +C_1 r^2 +C_2 r^4 +\ldots.
\end{align*}
We claim that $C_{-1}$ is vanishing. Otherwise, $\varphi(0)\ne 0$, which implies that $\varphi_0 (r) \eta_0$ is not well defined on the zero level. Recalling the expression of $\theta$, this contradicts against the fact that $\theta$ is well defined on $L$. Therefore, we find a global K\"ahler potential for $\omega$, which completes the proof. 
\end{proof}

\begin{rem} The invariant $\partial\overline{\partial}$-lemma does not hold for all line bundles. For instance, let $\alpha$ be a simple root with $\alpha=2 \lambda$, for instance, the line bundle $\OO(1)\rightarrow \CC\PP^1$. Consider the following invariant 1-form
\begin{align*}
\theta = r^2 \eta_0 + r^2 \eta_\alpha +r^2 \xi_\alpha.
\end{align*}
Then, we can check that $d\theta$ is an invariant exact real (1,1) form. However, comparing with (\ref{dform2}), there is no potential function for this form.
\end{rem}

\subsection{Proof of Theorem \ref{kf1}}\label{ss:mainproof}

Firstly, it is easy see that the invariant $dd^c$ lemma can be applied when $L$ is a negative line bundle over $X$. Notice that the homogeneous line bundle $L$ can shrink to the base manifold $X$. Furthermore, we can require the shrinking process to be $G$-equivalent. In particular, we have the following isomorphism between cohomology groups. 
$$p^*: \ H^{1,1}_{G}(X) \ \cong \ H^{1,1}_{G} (L),$$
which implies that all invariant K\"ahler classes of $L$ arise from the invariant K\"ahler classes of $X$. In each invariant K\"ahler class of $X$, there exists exactly one invariant K\"ahler form, which follows directly from $dd^c$-lemma on $X$. Given an invariant K\"ahler form $\omega$ on $L$, by previous discussion, there exists an invariant K\"ahler form $\omega_X$ on $X$ such that $[\omega]=p^*[\omega_X]$  then, by Propositions \ref{ddbar3}, there exists a smooth function $\Phi$ defined on $L$ such that,
\begin{align*}
\omega= p^* \omega_X + dd^c \Phi.
\end{align*}
Hence, we complete the proof of theorem A.

\section{Momentum profiles and the classification of invariant cscK metrics}\label{sec:mom}

This section is dedicated to the proof of Theorem \ref{main}. Our main tool of this section comes from \cite{hwang2002momentum}, in which the momentum construction is applied to investigate the Calabi ansatz. In Section \ref{calabisec}, we will build up momentum construction by tracking the idea in \cite{hwang2002momentum}. Then, momentum construction will be applied in Section \ref{calabisec1} to get the existence of $G$-invariant constant scalar curvature K\"ahler metric. We also prove uniqueness and determine the asymptotics of these metrics. This proves Theorem \ref{main}.

\subsection{Momentum construction of Calabi ansatz} \label{calabisec}
Let $(X,\omega_X)$ be a K\"ahler manifold and $p:(L,h)\rightarrow (X, \omega_X)$ be a holomorphic line bundle of $M$ with Hermitian metric $h$. Let $t$ be the logarithm of fibre norm function related to $h$; i.e., given a local line bundle coordinate chart $(u,v)$, where $u$ represents the fibre coordinate, $t=\log [r(u)^2]:=\log h(u,\overline{u})$. Then, K\"ahler metrics arise from Calabi Ansatz is given by
\begin{equation}\label{calabime}
\omega=p^*\omega_X+\frac{1}{2}dd^c f(t),
\end{equation}
where $f$ is a smooth function of one real variable. Actually, according to theorem A, all invariant K\"ahler metrics comes from Calabi ansatz (\ref{calabime}). 

It is well-known that the problem of prescribed scalar curvature is equivalent to a fourth order PDE of potential function, specially, in this case, a fourth order ODE. However, A.D. Hwang and M.A. Singer comes up the method of momentum profile in \cite{hwang2002momentum} by which it can be reduced to be a second order ODE, as the curvature formula related to momentum profile is of second order. In the following, we study the momentum profile associated with the given original data. 

The K\"ahler metric, $\omega$, arising from Calabi ansatz (\ref{calabime}) may not exist in the whole line bundle $L$. For instance, in some cases, it might be blowing up in a finite domain with respect to fibre coordinate. In the sequel, we use $L'$ to denote the possible existence region of $L$ for $\omega$.  Notice that K\"ahler metric $\omega$  constructed in (\ref{calabime}) admits a natural Killing field $X_0$, which generates a circle action on each fibre and can be written in local bundle coordinates,
\begin{equation}\label{killing}X_0=2\im (\overline{u}\frac{\partial}{\partial\overline{u}})=r J\frac{\partial}{\partial r}.\end{equation} In point view of symplectic geometry, $X$ generates a Hamiltonian action on $(L,\omega)$ by
\begin{equation}\label{momeq}i_{X_0}\omega=-d\tau.\end{equation}
At each point on $p \in X$, there exists a coordinate chart around $p$ such that $\partial\log h |_p=\overline{\partial} \log h|_p=0$ and we call this a canonical chart of hermitian metric at $p$.  By computing $\omega$ at each point in a canonical chart,
\begin{equation}\label{calabime2}
\omega=p^*\omega_X+f'(t)\frac{1}{2}dd^c\log h+f''(t)i\frac{du\wedge d\overline{u}}{|u|^2}
\end{equation}
and inserting (\ref{calabime2}) and (\ref{killing}) into (\ref{momeq}), we have $\tau=f'(t)$. Let the interval $I$ be the image of moment map $\tau$. Noting that $||X||_\omega$ is a constant along each level of $\tau$, we can define the function $\varphi: I\rightarrow \RR_{\geq 0}$ by factoring through $\tau$,
\begin{equation*}\varphi(\tau)=\frac{1}{2}||X_0(\tau)||^2_{\omega}.\end{equation*}
The interval $I$ together with the function $\varphi$ is called \textit{momentum profile} related to $(L',\omega)$. The essential relation between $\varphi(\tau)$ and the potential $f$ is given by,
\begin{equation}\label{legendra}\varphi(\tau)=\frac{1}{2}\omega(X_0, JX_0)=-\frac{1}{2}JX_0(f'(t))=-\frac{1}{2}f''(t) \cdot JX_0(\log r^2)=f''(t)\end{equation}

Also by observing (\ref{calabime2}), the positivity of K\"ahler form implies the following two things: $f$ is a convex function; hence the moment map $\tau=f'(t)$ induces a Legendre transformation from $t$ to $\tau$. Moreover, if we denote $\gamma=-i\partial\overline{\partial} \log h$, the positivity of $\omega$ also requires $\omega-\tau \gamma$ to be positive. An interval, $I$, is defined to be a \textit{momentum interval} if for all $\tau\in I$, $\omega(\tau)=\omega-\tau \gamma$ is positive. In the following, we shall reconstruct $(L',\omega)$ by momentum profile $(I,\varphi)$ with a momentum interval $I$  and $\varphi: I\rightarrow \RR_{\geq 0}$.  

Based on the inverse Legendre  transformation, we can rebuild the K\"ahler metric $\omega$ explicitly by momentum profile $(I, \varphi)$ as follows. Let $(0,b)$ be the interior of $I$ with $b\leq \infty$ and fix $\tau_0\in I$
\begin{enumerate}
\item[(a)] {Fibre domain:} 
Let $T$ be the defining domain of $f(t)$ with $T^\circ =(t_1,t_2)$, then,
\begin{align*}
 t_1=\lim_{\tau\rightarrow 0+}\int_{\tau_0}^\tau \frac{dx}{\varphi(x)}\quad \text{ and }\quad  t_2=\lim_{\tau\rightarrow b-}\int^\tau_{\tau_0}\frac{dx}{\varphi(x)}.
\end{align*}
\item[(b)] {Potential function:} noting that $t$ and $\tau$ are related by $\displaystyle t=\int^{\tau(t)}_{\tau_0} \frac{dx}{\varphi}$, then $f(t)$ is given by data $(I,\varphi)$
\begin{align*} 
f(t)=\int^{\tau(t)}_{\tau_0}\frac{xdx}{\varphi(x)}.
\end{align*}
\item [(c)]{Fibre metric:} The metric $\omega$ induces the metric on each fibre in terms of coordinate $u$, $\displaystyle g_{\text{fibr}e}$ and the $\omega$-distance between $\tau_0$ level and $\tau(t)$ level, $s(t)$
\begin{align*}
g_{\text{fibre}}=\varphi(\tau)\bigg|\frac{du}{u}\bigg|^2, \qquad  s(t)=\int_{\tau_0}^{\tau(t)}\frac{dx}{2\sqrt{\varphi(x)}}.
\end{align*}
\end{enumerate}
where the formulas in (b), (c) can be obtained by change the variable though Legendre transformation. In  Summary, the relationship between the K\"ahler metric data and the momentum profile can be shown in the following commutative graph. Let the momentum profile $(\varphi,I)$ determine a K\"ahler metric $\omega$ in (\ref{legendra}) on subbundle $L'\subset L$, 
\begin{align*}
\xymatrix{&&L'\ar^{\ t}_{\text{log-Hermitian metric}\quad \ }[lld]\ar_\tau^{\quad\text{moment map}}[rrd]&&\\T\ar^{f'}_{\text{Legendra transformation}}[rrrr]\ar_{f''}[rrd]&&&&I\ar^\varphi[lld]\\&& \RR_{\geq 0}}
\end{align*}

The next step is to work out the curvature formula in terms of momentum profiles.To make curvature formula fit in the momentum profile, define $(\tau,p): L'\rightarrow I\times X$. Here are some notations that will be used in the following: 
\begin{enumerate}
\item[$\bullet$] Let $\omega_\varphi$ represent the K\"ahler metric constructed by momentum profile $(\varphi,I)$, and we can rewrite (\ref{calabime2}) in terms of $\tau$,
\begin{equation}\label{calabime3}
\begin{split}
\omega_\varphi &=p^*(\omega_X-\tau\gamma)+\varphi(\tau)\frac{idu\wedge d\overline{u}}{|u|^2}\\
&=p^*\omega_X(\tau)+\varphi(\tau)\frac{idu\wedge d\overline{u}}{|u|^2}
\end{split}
\end{equation}
\item[$\bullet$] Let $B$ denote the endomorphism $\omega_X^{-1}\gamma$, $\rho_X$ be the Ricci curvature form of $X$, define the following functions on $I\times X$, 
\begin{align*}
Q(\tau)&=\det(I-\tau B),\\
R(\tau)&=\tr[(I-\tau B)^{-1}(\omega^{-1}_X\rho_X)]. 
\end{align*}
\end{enumerate}
Then, the Ricci curvature, Laplacian and scalar curvature have the following representation in terms of momentum profile and notations above
\begin{enumerate}
\item[$\bullet$] The Ricci form of $\omega_\varphi$,
\begin{equation}\label{cur1}
\rho_\varphi=p^*\rho_X-i\partial\overline{\partial}\log \varphi Q(\tau)
\end{equation}
\item[$\bullet$] Laplacian of a circle invariant smooth function $s$ with respect to $\omega_{\varphi}$: notice that $s$ can factor through $(\tau,p): L' \rightarrow I\times X$, especially, $s$ can be viewed a smooth function on $I \times X$, then
\begin{equation}\label{cur2}
\Delta_\varphi s=\Delta_{\omega_X(\tau)} s(\tau,\cdot)+\frac{1}{Q}\frac{\partial}{\partial \tau}\Big[\varphi Q\frac{\partial s}{\partial \tau}\Big].
\end{equation}
\item[$\bullet$] Scalar curvature $S_\varphi$,
\begin{equation}\label{cur3}
\begin{split}
S_\varphi&=S(\tau)-\frac{1}{Q}\frac{\partial^2}{\partial \tau^2}(\varphi Q)(\tau)\\
&=R(\tau)-\Delta_{\omega_{X(\tau)}}\log Q(\tau)-\frac{1}{Q}\frac{\partial^2}{\partial \tau^2}(\varphi Q)(\tau)
\end{split}
\end{equation}
\end{enumerate}
The formula (\ref{cur1}), (\ref{cur2}), (\ref{cur3}) can be easliy checked in local bundle coordinates. For instance, the formula (\ref{cur2}) follows by observing that if we write coordinates $(z_1,\ldots, z_{n-1}, u)$ where $\{z_1,\ldots, z_{n-1}\}$ is the bundle adapted coordinates of $X$ such that $\partial_i \log h= \partial_{\overline{i}}\log h=0$, 
\begin{align}
\Delta_\varphi s 
& = g^{i\overline{j}}_\tau \partial_i \partial_{\overline{j}} s(\tau,\cdot)
+ \varphi(\tau)\big( g^{i\overline{j}}_\tau \partial_i \partial_{\overline{j}} \log h \big) \frac{\partial s}{\partial \tau}
+ \frac{\partial}{\partial \tau} \Big( \varphi \frac{\partial s}{\partial \tau} \Big) \nonumber
\\& = \Delta_{\omega_{X}(\tau)}s(\tau,\cdot) - \varphi(\tau) ( \tr_{\omega_X(\tau)}\gamma ) \frac{\partial s}{\partial \tau}+ \frac{\partial}{\partial \tau} \Big( \varphi \frac{\partial s}{\partial \tau} \Big) . \label{cur4}
\end{align}
And 
\begin{align}
\tr_{\omega_X(\tau) \gamma} = \tr [(I-\tau B)^{-1} B] = -\frac{\partial}{\partial \tau} \det(I-\tau B)= -\frac{\partial Q}{\partial \tau}. \label{cur5}
\end{align}
By inserting (\ref{cur4}) into (\ref{cur5}), we get (\ref{cur2}). To view (\ref{cur3}) as an ODE with prescribed scalar curvature $S_\varphi$, we still need the following \textit{$\sigma$-constant} condition,
\begin{define}
The data $\{(L,h), (X,\omega_X)\}$ with moment map $(\tau,I)$ is said to be $\sigma$-constant, if
\begin{enumerate}
\item[(a)] $B=\omega_X^{-1}\gamma$ has constant eigenvalues on $X$.
\item[(b)] $\omega_{X}(\tau)$ has constant scalar curvature for each $\tau\in I$ 
\end{enumerate}
\end{define}
In the case of $\sigma$-constant, $ Q(\tau)$ is a polynomial in $\tau$ and $\Delta_{X(\tau)}\log Q(\tau)=0$. Then, we can assume that $S(\tau)=R(\tau)=P(\tau)/Q(\tau)$ for some polynomial $P$ in $\tau$, based on the definition of $R(\tau)$. Therefore, we reduce the problem of prescribed scalar curvature to a second order ODE,
\begin{equation}\label{coreode}
(\varphi Q)''+QS_{\varphi}=P.
\end{equation}

\subsection{Classification of $G$-invariant cscK metrics } \label{calabisec1}

In this subsection, we always assume $\dim_\CC L= n$. To prove the classification Theorem \ref{main} , we would like to apply the technique discussed in Section \ref{calabisec}. According to the $G$-symmetry of $X$ and its line bundle $L$, $\omega$, $\gamma$ are $G$-invariant (1,1) forms on $X$. Hence, the data $\{(L_\lambda,h),(X,\omega_X)\}$ satisfy the conditions of $\sigma$-constant. 

In the sequel, we shall compute the explicit formula of the polynomials $Q(\tau)$, $P(\tau)$ in terms of $\omega_X$ and corresponding weight $\lambda$. Recall the formulas (\ref{2f2}) and (\ref{chern2}) and $\omega_X$, $\gamma$ can be expressed in terms of $dz_\alpha=\eta_\alpha +i \xi_\alpha$ and $d\overline{z}_\alpha= \eta_\alpha - i\xi_\alpha$, 
\begin{align*}
\omega_X = \frac{i}{2} C_{\alpha,S}\ dz_\alpha \wedge d\overline{z}_\alpha,\qquad \gamma=-i\partial\overline{\partial} \log h= \frac{i}{4}\sum_{\alpha\in D^+}( \lambda , \alpha )\ dz_\alpha \wedge d\overline{z}_\alpha,
\end{align*} 
where $S\in \sff$ such that $C_\alpha(S)>0$. Then, the matrix $B$ is diagonal, then $Q(\tau)$ has the following expression,
\begin{align*}
Q(\tau)=\det(I-\tau B)=\prod_{\alpha \in D^+}\bigg[1-{\tau} \frac{(\lambda, \alpha)}{2 C_{\alpha,S}} \bigg],
\end{align*}
Since the Ricci curvature $\rho_X$ has the expression, 
\begin{align*}
\rho_X = \frac{i}{4} \sum_{\alpha \in D^+} (\alpha, \delta)\ dz_\alpha \wedge d\overline{z}_\alpha 
\end{align*}
Then, 
\begin{align*}
R(\tau)= \tr \big[(I-\tau B)^{-1}\ \omega_X^{-1} \rho_X \big]= \sum_{\alpha \in D^+ } \frac{(\alpha,\delta)}{2C_{\alpha,S}-\tau(\lambda,\alpha)}. 
\end{align*}
Hence, the ODE (\ref{coreode}) can be rewrite as follows,
\begin{align}
\Big( \varphi \prod_{\alpha \in D^+} \big( 2 C_{\alpha,S} - \tau (\lambda,\alpha) \big) & \Big)'' + S_\varphi \prod_{\alpha\in D^+}\big( 2 C_{\alpha,S}- \tau (\lambda, \alpha) \big) \nonumber
\\ &= \prod_{\alpha \in D^+} \big( 2 C_{\alpha,S}- \tau (\lambda,\alpha) \big) \sum_{\alpha \in D^+}\frac{(\alpha,\delta)}{2 C_{\alpha,S}-\tau (\lambda, \alpha)}  \label{coreode1}
\end{align}
To determine the initial data of the ODE (6.11), we need to apply the following completeness proposition in (\cite{hwang2002momentum}, Proposition 2.2, Proposition 2.3).
\begin{pro}\label{completeness}\textup{(A.D. Hwang, M.A. Singer)} 
Let $(I,\varphi)$ be a given momentum profile. Then the associated fibre metric is complete if and only if the following conditions hold at each endpoint of $I$
\begin{enumerate}
\item[$\bullet$] Finite Endpoints: $\varphi$ satisfies one of the following conditions,
\begin{enumerate}
\item[(i)] $\varphi$ vanishes to first order with $|\varphi'|=1$; or
\item[(ii)] $\varphi$ vanishes to order at least two.
\end{enumerate}
\item[$\bullet$] Infinite Endpoints: $\varphi$ grows at most quadratically, i.e., $\varphi\leq K\tau^2$.
\end{enumerate}
\end{pro}
And the corresponding $(L',\omega)$ behaves differently under different decay conditions provided in Proposition \ref{completeness}. We conclude the corresponding relations in the table \ref{class}, where we consider finite ends at $\tau=0$ and infinite ends as $\tau \rightarrow \infty$. The proof of these bundle behaviors directly follows from the reconstruction of the data $(L',\omega)$ by moment profile $(\varphi,I)$, (a)--(c). 

\begin{table}[h]
\centering
\begin{tabular}{|c|c|c|c|}
\hline
Type of Ends & Decay (Growth)  & Fibre Range $(t)$ & Distance to \\
& conditions & & Ends ($\omega$)\\
\hline
finite ends & $\varphi=0, \quad \varphi'=1$, & $[-\infty, t_0]$  & finite\\
\hline
finite ends & $\varphi=\varphi'=0$ & $(-\infty, t_0]$  & infinite\\
\hline
infinite ends & $C\tau^{1+\epsilon}\leq\varphi\leq K \tau^2$ & $[t_0,t_{\text{end}})$, ($t_{\text{end}}<\infty$) & infinite\\
\hline
infinite ends & $\varphi\leq C \tau$ & $[t_0,\infty)$ & infinite\\
\hline
\end{tabular}
\caption{Behaviors of $(L',\omega)$}
\label{class}
\end{table}

To fit in our cases, we define the momentum interval $I=[0,b)$ with $b\leq +\infty$. The reason we take the left ends to be $0$ is to ensure $\omega|_{X}=\omega_{S}$. Assume that the corresponding scalar curvature of $\omega_\varphi$ is constant. Combining with Proposition \ref{completeness} and table \ref{class}, we shall solve the ODE with initial condition $\varphi(0)=0$, $\varphi'(0)=1$ and $S_\varphi=C$. It is obvious that there is a unique solution $\varphi$ satisfies (\ref{coreode1}) and the initial condition.To understand the behavior of these metrics, we shall investigate the solution $\varphi$. In the sequel, the weight $\lambda$ is supposed to be semi-negative, thus $(\lambda, \alpha)\leq 0$. 

\subsubsection{Case 1: constant negative scalar curvature $(C<0)$.} In this case, the ODE (\ref{coreode1}) implies the following integral formula for the first derivative of $\varphi$
\begin{align*}
\Big( \varphi \prod_{\alpha \in D^+} \big( 2 C_{\alpha,S} - &\tau (\lambda,\alpha) \big)  \Big)'
=\int_{0}^\tau - C \prod_{\alpha\in D^+}\big( 2 C_{\alpha,S}- t (\lambda, \alpha) \big) dt
 \\ &+\int_0^{\tau}\prod_{\alpha \in D^+} \big( 2 C_{\alpha,S}- t (\lambda,\alpha) \big) \Big(\sum_{\alpha \in D^+}\frac{(\alpha,\delta)}{2 C_{\alpha,S}-t (\lambda, \alpha)} \Big) dt+ \prod_{\alpha\in D^+} 2 C_{\alpha,S}.
\end{align*}
By assumption, both $C_{\alpha,S}$ and $-C$ are positive and $-(\lambda, \alpha)$ is nonnegative. Hence, in the momentum interval $I=[0,b)$, $\varphi$ is strictly increasing with initial value $\varphi(0)=0$. Since $\varphi$ is a rational function with degree two, referring to table \ref{class},  the corresponding metric is complete in the fibre coordinate $(t; -\infty \leq t < t_{\text{end}}<\infty)$ which implies that the metric $g_\varphi$  blowup to infinity when the log-Hermitian coordinate approach to $t_{\text{end}}$. 

In fact, the K\"ahler metric defined by $\varphi$ when $C<0$ is asymptotic to hyperbolic metric along each fibre. Since the standard hyperbolic metric in complex disk can be written as
\begin{align} \label{hypmod}
g_H = d\rho^2 + (\sinh \rho )^2 \eta_0,
\end{align}
where $\rho$ is the hyperbolic distance to origin. The model metric (\ref{hypmod}) can be applied to describe the asymptotic behavior of K\"ahler metric along each fibre. In particular, if there exists a radial function $\rho$ defined outside a compact neighborhood $K$ of zero level and reach to infinity approaching to boundary such that the metric can be represented in outside the compact neighborhood $K$,
\begin{align}\label{hypasym}
g = d\rho^2 + e^{2\rho} g(\rho),
\end{align}
where $g(\rho)$ approaches to $\eta_0^2$ at boundary. To fit in the asymptotic condition, define $\rho$ to be a distance function with respect to K\"ahler metric $\omega_\varphi$. Let $S_\varphi= C < 0$ in ODE (\ref{coreode}) and the leading term of $\varphi(\tau)$ is $-C\tau^2/(n^2 + n)$. If we write $\alpha_{n,C}^2 = -4C/(n^2+n)$, then the fibre metric can be written as 
\begin{align*}
g_{\text{fibre}} = d\rho^2 + \frac{1}{\alpha^2_{n,C}} \phi(\rho) \eta_0^2, \quad \phi(\rho) =    e^{2 \alpha_{n,C} \rho} + c_0 + c_1 e^{-2 \alpha_{n,C} \rho} + \ldots.
\end{align*} 
Hence, let $s=\alpha_{n,C} \rho$, 
\begin{align*}
g_{\text{fibre}} = \frac{1}{\alpha_{n, C}^2} \big[ds^2 + \phi(s) \eta_0^2\big].
\end{align*}
Therefore, we conclude that $g_{\text{fibre}}$ is asymptotic to a hyperbolic metric.

\subsubsection{Case 2: zero scalar curvature $(C=0)$.} In this case, the unique solution satisfies
\begin{align*}
\Big( \varphi \prod_{\alpha \in D^+} \big( &2 C_{\alpha,S} - \tau (\lambda,\alpha) \big)  \Big)'=
 \\ &\int_0^{\tau}\prod_{\alpha \in D^+} \big( 2 C_{\alpha,S}- t (\lambda,\alpha) \big) \sum_{\alpha\in D^+}\frac{(\alpha,\delta)}{2 C_{\alpha,S}-t (\lambda, \alpha)}  dt+ \prod_{\alpha\in D^+} 2 C_{\alpha,S}.
        \end{align*}
Likewise, $\varphi$ is a strictly increasing function with initial value $\varphi(0)=0$. However, in this case, the degree of $\varphi$ is one. According to table \ref{class}, $\omega_\varphi$ is well-defined over the whole bundle $L_\lambda$. Besides, $\omega_\varphi$ is asymptotic to a K\"ahler cone. Indeed, by solving the ODE (\ref{coreode1}), the leading coefficient of the solution is given as follows,
\begin{align*}
i_{\lambda, X} = \frac{1}{(n-1)n} \sum_{\alpha \in D^+} \frac{(\alpha, \delta)}{(\alpha,-\lambda)}.
\end{align*}
We call $i_{\lambda, X} $ the \textbf{metric index} of line bundle of $(L_\lambda, X)$. Recall the K\"ahler metric associated to $\varphi$ is given as in (\ref{calabime3}). Let $g_X(\tau)$, $g_X$, $g_\gamma$ be the metric corresponding to $\omega_X(\tau) = \omega_X - \tau \gamma$, $\omega_X$,  $-\gamma$ respectively then
\begin{align*}
g_\varphi &= g_X(\tau) + 2\varphi(\tau) \eta_0^2 + 2 \varphi (\tau) \mu^2,\\
& =g_X + \tau g_\gamma +2 \varphi(\tau) \eta_0^2+ 2\varphi(\tau)\mu^2.
\end{align*}
where $\mu$ and $\eta_0$ are the dual of $r\partial/\partial r$ and $X_0$ respectively. Then, on each level set of $L$, $M(\tau)$, there is a metric induced by $g_\varphi$, denoted by $g(\tau)$,
\begin{align*}
g_{M(\tau)}= g_X + \tau g_\gamma +2 \varphi(\tau) \eta_0^2.
\end{align*}
Let $C_\lambda$ be the cone associated with $L_\lambda$ by collapsing the base manifold $X$.  Then, we should determine the radial function $l$ of cone $C_\lambda$ such that the scalar-flat K\"ahler metric is asymptotically conical to $A i\partial\overline{\partial} l^2$, where $A$ is the constant coefficient and can be canceled by rescaling. Based on the discussion in Section \ref{calabisec}, we have the following relationship between $\tau$ and $t$,
\begin{align*}
t=\int_{\tau_0}^{\tau(t)} \frac{dx}{\varphi(x)} = \int_{\tau_0}^{\tau(t)} \frac{dx}{i_{\lambda, X} x}+\frac{a_1 dx}{x^2}+\ldots = a_0+\frac{1}{i_{\lambda, X}}\log \tau - \frac{a_1}{\tau}+\ldots 
\end{align*}
where the second equality is just the Taylor expansion of $1/\varphi(x)$. Taking exponential and solve for $\tau$, we can see that $\tau$ admits the following expansion at infinity,
\begin{align}
\tau= b_1 r^{2i_{\lambda, X}} + b_0 + b_{-1} r^{-2 i_{\lambda, X}}+\ldots
\end{align}
Now, let $l=r^{i_{\lambda,X}}$, then the model K\"ahler metric over $C_\lambda$ is defined by the radial function $l$,
\begin{align*}
\omega_{\textup{mod}}=
i\partial\overline{\partial} l^2= - i_{\lambda, X} l^2 \gamma + 2 i_{\lambda, X}^2 l^2  \mu \wedge \eta_0
\end{align*}
Rewrite the model K\"ahler form in terms of metric,
\begin{align*}
g_{\textup{mod}} &= i_{\lambda, X} l^2 g_\gamma+ 2 i_{\lambda, X}^2 l^2 \eta_0^2 + 2 i_{\lambda, X}^2 l^2 \mu^2\\
&= l^2 (i_{\lambda, X}  g_\gamma+  2 i_{\lambda, X}^2 \eta_0^2 ) + 2 dl^2
\end{align*}
And the metric can be represented by $dl$ as follows,
\begin{align*}
g_\varphi & = l^2\Big( \frac{1}{l^2} g_X +  \frac{\tau}{l^2} g_\gamma +\frac{2 \varphi(\tau)}{l^2} \eta^2_0 \Big) + \frac{2\varphi(\tau)}{i_{\lambda, X}^2 l^2} dl^2\\
& = {b_1} \big[ l^2 (i_{\lambda, X} g_\gamma+ 2 i^2_{\lambda, X} \eta^2_0) + dl^2 \big]+ O(l^{-2})\\
& = b_1 g_{\textup{mod}} +O(l^{-2})
\end{align*}
Therefore, all scalar-flat K\"ahler metrics on $L_\lambda$ are asymptotically conical to $(C_\lambda, g_{\textup{mod}})$. In general, $g_\varphi$ decays to $g_{\textup{mod}}$ by order $-2$, which can be improved in some special cases. For instance, let the base metric $\omega_X $ be equal to the curvature form $\gamma$, then, the similar calculation shows that the metric $g_\varphi$ with scalar-flat curvature decays to order $-2n+2$.

\subsubsection{Case 3: constant positive scalar curvature $(C > 0)$.} If we write 
$$ \Phi (\tau) = 
\varphi (\tau) \prod_{\alpha \in D^+} \big( 2C_{\alpha,S} - \tau (\lambda,\alpha) \big),$$
In the sequel, we investigate the behavior of $\Phi(\tau)$ when $\tau \geq 0$. The following lemma shall be applied.

\begin{lem} \label{poly}
Consider the polynomial,
\begin{align*}
p(x)=c\prod_{k=1}^n (x+a_k)-\sum_{k=1}^n b_k \prod_{1\leq i\leq n,i\ne k}(x+ a_i),
\end{align*}
where $c$, $a_k$, $b_k$, $1\leq k\leq n$ are all positive. Then $p(x)$ has at most one positive root.
\end{lem}
\begin{proof} Define a new polynomial,
\begin{align*}
q(x)=\frac{1}{\prod_{k=1}^n (x+a_k)} \cdot p(x) = c-\sum_{k=1}^n \frac{b_k}{x+a_k}
\end{align*}.
Notice that $q(x)$ has the same positive roots as $p(x)$. It suffices to investigate the positive root of $q(x)$. Since we have,
\begin{align*}
q'(x)=\sum_{k=1}^n \frac{b_k}{(x+a_k)^2} >0
\end{align*}
Hence, $q(x)$ is an increasing function with $\lim_{x\rightarrow \infty} q(x)= c>0$. In conclusion, we have the following result,
\begin{enumerate}
\item[(i)] If $\displaystyle c-b_k / a_k >0$, there exists no positive root.
\item[(ii)] If $\displaystyle c- b_k / a_k <0$, there exists exact one positive root.
\end{enumerate}
We complete the proof.\end{proof}
Noting that $-\Phi''(\tau)$ is a polynomial of type in Lemma \ref{poly}, $\Phi''(\tau)$, $\tau \geq 0$, is a decreasing function with at most one positive root. Then, according to the initial condition of $\varphi$, we have $\Phi'(0)=\prod_{\alpha\in D^+}2 C_{\alpha,S}>0$. Since the leading coefficient $\Phi'(\tau)$ is negative, then there exists exact one positive root for $\Phi'(\tau)$.  
Then, based on the fact that $\Phi'(\tau)$ is positive near $0$ and the initial condition, $\Phi(0)$, we deduce that $\Phi(\tau)$ is increasing first and decreasing to negative infinity as the leading coefficient of $\Phi(\tau)$ is also negative. That is to say, there exists exact one positive simple root of $\Phi(\tau)$, so as $\varphi(\tau)$. Let $b$ be the only positive root of $\varphi$, then the momentum interval in this case is $[0,b)$. Referring to table \ref{class}, the metric $\omega_\varphi$ is defined over the whole bundle of $L_\alpha$. $\omega_\varphi$ can be completed by adding a divisor at infinity. The completed space is smooth if and only if $\varphi'(b)=-1$, otherwise each fibre has a singularity at infinity.

To investigate the singularity, consider the case that $\varphi'(0) \ne 1$ (assume $\varphi'(0)=a$), then the expansion of $\varphi(\tau)$ at $\tau=0$ can be written as
\begin{align*}
\varphi(\tau) = a\tau + O(\tau^2).
\end{align*}
According to reconstruction of momentum profile, we can get an expansion of $\varphi(\tau)$ in terms of $r$,
\begin{align*}
\varphi(\tau)= c_1 r^{2a} + O(r^{4a}).
\end{align*}
Recall that the fibre metric of $\omega_\varphi$ is given as $\displaystyle g_{\text{fibre}} = \varphi(\tau)\frac{|du|^2}{|u|^2}$. Combining with the expansion formula of $\varphi(\tau)$, the fibre metric can be represented as follows: let $l=r^a$, then
\begin{align*}
   g_{\text{fibre}} = C\big[(dl)^2 + a^2 l^2 (\eta_0)^2 \big] +  O(l^2).
\end{align*}
Hence, along each fibre the singularity just looks like a cone around the infinity.

\bibliographystyle{plain}
\bibliography{Reference}
\end{document}